\documentclass{amsart}

\usepackage{enumitem}
\usepackage[T1]{fontenc}
\usepackage[utf8]{inputenc}
\usepackage[english]{babel}
\usepackage{amsmath}
\usepackage{amsthm}
\usepackage{amssymb}
\usepackage[a4paper, margin=1in]{geometry}
\usepackage{comment}

\usepackage{subcaption}

\usepackage{graphicx}
\usepackage{tikz}
\usetikzlibrary{shapes.misc}
\usepackage{float}
\usepackage{amsopn}

\DeclareMathOperator{\Ker}{Ker}

\usepackage[unicode=true,pdfusetitle,
bookmarks=true,bookmarksnumbered=false,bookmarksopen=false,
breaklinks=false,pdfborder={0 0 1},backref=false,hidelinks]
{hyperref}

\makeatletter

\numberwithin{equation}{section}
\numberwithin{figure}{section}

\makeatother

\newtheorem{theorem}{Theorem}[section]

\newtheorem{definition}[theorem]{Definition}
\newtheorem{example}[theorem]{Example}

\newtheorem{lemma}{Lemma}[section]

\newtheorem{remark}[theorem]{Remark}

\begin{document}
	
	\title[test]{Towards nonlinearity. The $p$-regularity theory. Applications and developments. }
	
	\author{
		E. Bednarczuk$^{1}$ 
	}
	
	\author{
		O. Brezhneva$^{2}$ 
	}

	\author{
		K. Le\'{s}niewski$^{3}$ 
	}
	\author{
		A. Prusi\'{n}ska$^{4}$ 
	}
	\author{
		A. Tret'yakov$^{4,5}$ 
	}
	\thanks{$^1$ Warsaw University of Technology, Poland
  ;ewa.bednarczuk@pw.edu.pl
	}
	\thanks{$^2$ Department of Mathematics, Miami University, Oxford, Ohio, USA; brezhnoa@miamioh.edu}
	\thanks{$^3$ System Research Institute, Polish Academy of Sciences, Poland;krzysztof.lesniewski@ibspan.waw.pl}
	\thanks{$^4$  Faculty of Science, University of Siedlce, Poland; aprus@uws.edu.pl}
	\thanks{$^5$  Dorodnicyn Computing Center, Federal Research Center ``Computer Science and Control," Russian Academy of Sciences, Moscow, 119333 Russia; tret@uws.edu.pl\\
	The work of the fourth author was supported also by the Russian Foundation for Basic Research (project No.~17-07-00510) and the RAS Presidium Program (program 27).}

	\maketitle
	
\smallskip

\begin{flushright}
    To the memory of our friends and colleagues \\
    Asen Dontchev (1948-2021) and Jerrold E. Marsden (1942-2010)\\
\end{flushright}
	\begin{abstract}
	
We present recent advances in the analysis of nonlinear equations with singular operators and nonlinear optimization problems with constraints given by singular mappings. The results are obtained within the framework of $p$-regularity theory, which has developed successfully over the last forty years. We illustrate the theory with its applications to degenerate problems in various areas of mathematics. In particular, we address the problem of describing the tangent cone to the solution set of nonlinear equations in a singular case. The structure of $p$-factor operators is used to propose optimality conditions and construct numerical methods for solving degenerate nonlinear equations and optimization problems.

The methods presented in the paper can be considered as the first numerical approaches targeting solutions of degenerate problems, such as the Van der Pol differential equation, boundary-value problems with a small parameter, partial differential equations where Poincaré's method of small parameter fails, nonlinear degenerate dynamical systems, and others.
There are various practical applications for the theory of $p$-regularity, including structural engineering, composite materials, and material design. For instance, the theory can be applied to analyze the behavior of materials with irregular or complex properties. By considering higher-order derivatives, it becomes possible to model and predict the response of materials to external forces, such as stress or temperature variations.
In geophysics, the $p$-regularity theory can be utilized to analyze and interpret complex data obtained from seismic surveys, gravity measurements, or electromagnetic surveys. The theory also finds applications in the analysis of nonlinear differential equations arising in control systems, geometric and topological analysis, biomechanics, and many other fields.
The theory offers a valuable approach to analyzing and understanding nonlinear phenomena in situations where classical regularity assumptions fail. This allows for a more comprehensive and nuanced understanding of complex systems. By incorporating higher-order derivatives and adapting the analysis to handle irregularities, $p$-regularity techniques provide valuable tools for understanding and characterizing complex phenomena in these settings.
	\end{abstract}

  {\bf AMS: } 47J05, 49M15, 65D05, 34L30, 49K27, 65K10, 90C30

\section{Introduction}
\label{intro}
Many fundamental results in nonlinear analysis and classical numerical methods in Banach spaces $X$ and $Y$ rely on the regularity of a mapping $F: X \to Y$ at some point $\bar{x} \in X$. The regularity of a Fr'echet differentiable mapping $F$ is commonly understood as the surjectivity of its Fr'echet derivative $F'$. However, a growing number of applications in areas such as partial differential equations, control theory, and optimization require the development of special approaches to deal with nonregular problems.

We introduce the theory of $p$-regularity, which originated in the 1980s with the aim of providing constructive tools for the analysis of nonregular problems. To date, the theory of $p$-regularity has found successful applications in various contexts and different areas of mathematics, as discussed in numerous papers. This paper highlights the most distinguished applications of the theory of $p$-regularity, with the goal of reviewing important results and indicating potential and promising directions for its future development and applications.

The theory of $p$-regularity, also known as higher-order regularity theory, offers a framework for studying nonlinear problems in situations where regularity assumptions are not satisfied. It focuses on utilizing higher-order derivatives to analyze and understand the behavior of mappings that are not onto or lack regularity.
	 
	 \begin{definition}
	[c.f. Definition 1.16 of \cite{ioffe}]
	\label{regular_point}
 Let $F:X\rightarrow Y$ be a continuously differentiable mapping from an open set $U\subset X$  of a Banach space $X$ into a Banach space $Y$. A vector $x\in U$ is called a {\em regular} point of $F$ if
$F'(x)$ maps $X$ onto the entire space $Y$, expressed as $\text{Im }F'(x)=Y.$ If $\text{Im }F'(x)\neq Y$, we refer to $x$ as a singular (nonregular, irregular, degenerate) point of $F$.	
\end{definition}

\subsection{Recollection of the fundamental results in the regular case}

	Regularity is a common assumption in many fundamental results of real and functional analysis, such as the inverse function theorem and the implicit function theorem. 
 In this section, we revisit some of these results. The theorems presented in this section have their roots in the following classical result.
	
	\begin{theorem}[Banach Open Mapping Principle,  \cite{Banach1932}, see also \cite{MR4393590}]
 \label{banach}
		Let $X$ and $Y$ be Banach spaces. For any linear
	and bounded single-valued mapping $A:X\rightarrow Y$, the following properties are
	equivalent:
	\begin{description}
	\item[(a)] $A$ is surjective.
	\item[(b)] $A$ is open.
	\item[(c)] There is a constant $\kappa>0$ such that for all $y\in Y$ there exists $x\in X$ with $Ax=y$ and
	$\|x\|\le\kappa\|y\|$.
\end{description}
\end{theorem}

The Banach open mapping principle can be extended to nonlinear mappings in many ways. One of those results is stated in the next theorem.

\begin{theorem} [Graves theorem, \cite{MR35398}] 
\label{th:graves}Let $X$ and $Y$  be Banach spaces and let $f$ be a continuous
function from $X$ to $Y$ defined in $B(0, \varepsilon)$ for some $\varepsilon>0$ with $f(0)=0$, where $B(x,t)$ stands for an open ball with radius $t$ and the center at $x$. Let $A$ be a continuous and linear operator from $X$ onto $Y$ and let $\kappa>0$ be the corresponding constant from Theorem \ref{banach} part (c). Suppose that there exists a constant $\delta<\kappa^{-1}$ such that
\begin{equation} \label{eq:graves}
\| f (x_{1})-f (x_{2})-A(x_{1}-x_{2})\|\le\delta\|x_{1}-x_{2}\|
\end{equation}
whenever $x_{1},x_{2}\in B(0, \varepsilon)$. Then the equation $y=f(x)$ has a solution $x\in B(0,\varepsilon)$ whenever $\|y\|\le c$, where $c=\kappa^{-1}-\delta$.
\end{theorem}

Note that the assumption of differentiability of $f$ at $0$ is not made, owing to the introduction of the concept of a strictly differentiable function several years after the publication of Graves' work \cite{MR35398}. Instead, the surjectivity of the operator $A$ is used. The proof of Theorem \ref{th:graves}, along with its reformulation in terms of a strictly differentiable function $f$, and related discussions can be found in \cite{dontchev1996graves}. For historical remarks, refer to \cite{Dontchev2019}.


Both the inverse function theorem and the implicit function theorem can be deduced from  Theorem \ref{Lyusternik-Graves} below; for details, see Theorem 1.20 in Section 1.2 of \cite{ioffe}.  We should also mention that some variants of the inverse function theorem can be immediately deduced from the implicit function theorem, see  Section \ref{sec:IFT} for details.

To state Theorem \ref{Lyusternik-Graves}, let us recall Definition 1.6 from \cite{ioffe}.

The {\em Banach constant} of a bounded linear operator $A$ between Banach spaces $X$ and $Y$, denoted by C(A), is defined as follows
\begin{equation} 
\label{constant} 
C(A)=\sup\{r\ge 0 \mid r {B}_{Y}\subset A({B}_{X})\}=\inf\{\|y\| \mid  y\not\in A({B}_{X})\},
\end{equation}
where ${B}_{X}$ and  ${B}_{Y}$ represent the unit balls in spaces $X$ and $Y$, respectively.


\begin{theorem}[Lyusternik–Graves Theorem,  \cite{ioffe}]
	\label{Lyusternik-Graves}
	Let $X$ and $Y$ be Banach spaces. Suppose that $F:X\rightarrow Y$ is 
	strictly differentiable and regular at $\bar{x}\in X$. 
	Then for any positive $r<C(F'(\bar{x}))$, there exists an $\varepsilon>0$ such that
	$$
	 {B}(F(x),rt)\subset F(B(x,t)),
	$$
	whenever $\|x-\bar{x}\|<\varepsilon$ and $0\leq t<\varepsilon$.
\end{theorem}

	For a thorough analysis of numerous  consequences of Theorem \ref{Lyusternik-Graves}  we refer the reader to the review paper by Dmitruk, Milyutin, and Osmolovskii \cite{DMO1980}, where the theorem is called the ``generalized Lyusternik theorem." Also in the monographs by Dontchev \cite{MR4393590}, Ioffe \cite{ioffe}, and Dontchev and Rockafellar \cite{MR3288139}, Theorem \ref{Lyusternik-Graves}  is called the ``Lyusternik-Graves theorem." From a more general point of view, the theorem is treated by  Dontchev and Frankowska in \cite{MR2736335, MR3059047}.

One of the consequences of Theorem \ref{Lyusternik-Graves} is the description of tangent vectors to the level set of a continuously differentiable mapping $F$ at regular points  (see  Section \ref{Sec:LT} below).

\subsection{Generalizations}	
	Since the early 1970s, due to theoretical interests and an increasing number of involved economic and industrial applications,  a vast literature was devoted to relaxing the surjectivity assumption of the derivative in the fundamental results, some of which are given above, while maintaining as much of their conclusions as possible. It is beyond the scope of this paper to provide an exhaustive survey of the existing generalizations of the theorems stated above. For the purpose of the paper, we can distinguish generalizations exploiting higher-order derivatives (see e.g. Frankowska  \cite{MR1204019, MR1019118}) and generalizations which attempt to relax the surjectivity assumption of the derivative without referring to higher-order derivatives  (see e.g. Ekeland \cite{MR2765512}, Hamilton \cite{Hamilton}, Bednarczuk, Le\'sniewski, and Rutkowski \cite{MR4223867}). The theory of $p$-regularity belongs to the first group of generalizations where higher-order derivatives are involved. 
	
	In this manuscript, we present the main concepts and results of the $p$-regularity theory, which has been developing successfully for the last forty years.  
	One of the main goals of the theory of $p$-regularity is to replace the operator of the first derivative, which is not surjective, by a special mapping that is onto.
	Nonlinear mappings analyzed within the framework of the theory of $p$-regularity are those for which the derivatives up to the order $p-1$ are not surjective at a given point $\bar{x}$,  where $p$ is a number ($p \geq 2$).
The main concept of the $p$-regularity theory is the construction of the $p$-factor-operator, which is surjective at the point $\bar{x}$ (see Definition \ref{4def1}). The special definition and the property of surjectivity of the  $p$-factor operator lead to generalizations of the fundamental results of analysis, including the implicit function theorem, and some classical numerical methods.  The $p$-factor operator  is defined in such a constructive way that it efficiently replaces the nonsurjective first derivative in a variety of situations. 
The structure of the $p$-factor-operator is used as a basis for analyzing nonregular problems and for constructing numerical methods for solving degenerate nonlinear equations and optimization problems.
We discuss those generalizations in this paper.

	 There are many publications that focus on the case of $p=2$ and use a 2-factor operator in a variety of applications.  
In this work, we consider a more general case of $p \geq 2$ and do not make some additional assumptions introduced and required in the publications of other authors.

In the framework of metric spaces, the concept that is related to the problems discussed in the present paper and attempting to generalize the classical results given above is the concept of metric regularity, see, for example,  \cite{MR4393590, ioffe}. 
For a function $f$ acting between Banach spaces $X$ and $Y$ and being strictly differentiable at the point $\bar{x}$, Corollary 5.3 in \cite{MR4393590} complements Theorem 5.1 in \cite{MR4393590}. 
It concludes that metric regularity at $\bar{x}$ for $f(\bar{x})$ is equivalent to surjectivity of the Fr\'echet derivative $Df(\bar{x})$ of $f$ at $\bar{x}$. 
In the case when $X= Y = \mathbb{R}^n$ this is the same as the nonsingularity of the Jacobian matrix $\nabla f(\bar{x})$.


The theory of $p$-regularity and the apparatus of $p$-factor operators makes it possible to create new methods of computational mathematics to solve such nonlinear problems of mathematical physics as the Van der Pol differential equation, boundary-value problems with a small parameter, partial differential equations where Poincaré's method of small parameter fails, nonlinear degenerate dynamical systems, and others.
This is associated with the proposed fundamentally new design (after Newton) of a numerical method for solving essentially (degenerate) problems, which is described in this paper. Moreover, the proposed approach will allow us to construct a new type of difference schemes of computational mathematics for solving problems of nonlinear mathematical physics, which are stable and  converge to a solution quickly.

This pertains to the numerical solution of nonlinear equations such as the nonlinear heat equation, Burgers equation, Korteweg-de Vries equation, Navier-Stokes equation, etc. It also enables us to reorganize Numerical Analysis in a novel way, with solutions obtained being related to the mentioned problems. All of this is applicable to the emerging prospects for developing new technologies and designs in Computational Mathematics for solving problems and models related to Artificial Intelligence, Optimization problems, dynamical systems, optimal control problems, etc. New opportunities are emerging for modeling and researching neural networks and creating new architectures for supercomputing.

It is crucial to emphasize the creation of fundamentally new and pioneering methods in computational mathematics. The resulting schemes were far from any previous designs, emerging from years of research into the structure of degeneracy  --- specifically, the structure of degenerate mappings and sets of solutions to degenerate systems. The analysis of these structures significantly differed from the analysis of linear problems, presenting entirely new forms of research objects.

Simultaneously, the previously hidden nonlinear world of unknown objects revealed an unexpectedly rich diversity. The class of possible new methods is explosively rich, featuring a rigid structure that is independent of small parameters, similar to regularization. These methods allow for adjustment depending on the problem being solved, as seen in the $p$-factor-operator, which depends on $h$: 
$$F^{(p)}(x^*)[h]^{p-1}.$$

Moreover, recent studies \cite{EBPT} revealed that the so-called ill-posed problems and essentially nonlinear ones are locally equivalent.
Therefore, many important problems such as inverse problems and others can be solved by using the p-factor method or the p-factor regularization. This is a new direction in mathematics and practical applications.

\subsection{Applications of the p-regularity}
The applications of the $p$-regularity theory extend across various fields, addressing nonregular or degenerate problems. In this section, we highlight a few examples.

The p-regularity theory plays a pivotal role in nonlinear analysis and critical point theory, especially in the study of nonregular or degenerate critical points of smooth functions. It provides tools for characterizing the behavior of critical points, understanding their indices, and establishing qualitative properties of critical point sets.

In the field of nonlinear optimization, mathematical programming, and variational analysis, the p-regularity theory can be utilized to tackle nonregular or degenerate objective functions or constraints. By incorporating higher-order derivatives and considering irregularities, the p-regularity techniques aid in the analysis of critical points, convergence properties of optimization algorithms, and the existence of solutions in constrained optimization problems. This results in improved algorithms, convergence analysis, and solution methods for optimization problems in various fields, including engineering, economics, and logistics. 

The application of the $p$-regularity theory extends to the study and analysis of nonlinear degenerate partial differential equations (PDEs). By considering higher-order derivatives and exploiting the nonregularity of the problem, the constructions of p-regularity help understand the behavior of solutions, establish existence results, and investigate qualitative properties such as stability, bifurcations, and concentration phenomena.


The p-regularity theory also finds applications in the analysis of nonlinear differential equations arising in control systems. By considering higher-order derivatives, it contributes to understanding controllability, stability, and solution behavior in nonregular or degenerate control systems, offering insights into the design and analysis of control strategies.

Moreover, the p-regularity theory has applications in geometric and topological analysis, particularly in studying nonregular or degenerate geometric objects. It can be used to analyze critical points, critical sets, and singularities of maps or geometric structures. By considering higher-order derivatives and the irregularity of the objects involved, the p-regularity techniques provide insights into their geometry, topology, and qualitative properties.

The p-regularity theory can be applied to analyze control systems or differential inclusions with nonregular or degenerate dynamics. By considering higher-order derivatives and the irregularities in the system, p-regularity techniques can help in understanding stability properties, controllability, and the existence of solutions in control problems with nonstandard regularity assumptions.

The p-regularity theory can be employed in the analysis of singular perturbation problems, where the regularity assumptions may not hold. By considering higher-order derivatives, p-regularity techniques can help in understanding the behavior of solutions near singular points or boundary layers, providing insights into the dynamics and asymptotic behavior of the system.

The p-regularity theory has applications in the study of phase transitions and critical phenomena in various physical and mathematical models. By incorporating higher-order derivatives, it can provide a more accurate description of the behavior near critical points, allowing for a better understanding of phase transitions and critical phenomena, such as in statistical mechanics or mathematical physics.

In the field of solid mechanics, the p-regularity theory can be used to analyze the behavior of nonlinear elastic and plastic materials. By considering higher-order derivatives, it can help in characterizing the stress-strain relations, determining the critical points, and understanding the onset of plastic deformation or material failure in nonregular or degenerate materials.

The p-regularity theory can be applied to nonlinear eigenvalue problems, where the regularity assumptions may not be satisfied. By incorporating higher-order derivatives, it can assist in studying the existence, multiplicity, and properties of solutions to nonlinear eigenvalue problems, offering insights into the behavior of eigenfunctions and eigenvalues in nonregular or degenerate settings.

In the field of biomechanics, the p-regularity theory can be utilized to analyze and model human movement. By considering higher-order derivatives, it becomes possible to account for irregularities, jerky motions, or non-smooth behaviors that may arise in complex human motions. This allows for a more accurate representation and analysis of movements in fields such as sports science, rehabilitation, and ergonomics.

Moreover, the theory can be applied to analyze the behavior of materials with irregular or complex properties. By considering higher-order derivatives, it becomes possible to model and predict the response of materials to external forces, such as stress or temperature variations. This finds applications in structural engineering, composite materials, and material design. In geophysics, the p-regularity theory can be used to analyze and interpret complex data obtained from seismic surveys, gravity measurements, or electromagnetic surveys. By incorporating higher-order derivatives, it becomes possible to characterize subsurface properties, detect anomalies, and improve the accuracy of geological models used in resource exploration or environmental assessment.

Furthermore, the $p$-regularity theory can be employed in image and signal denoising applications. By considering higher-order derivatives, it becomes possible to exploit the intrinsic smoothness or regularity of signals or images to remove noise or unwanted artifacts. This has applications in areas like medical imaging, image restoration, and signal processing.

In financial risk management, the theory of $p$-regularity can be  applied to assess and model the behavior of complex financial instruments or portfolios. By considering higher-order derivatives, it becomes possible to capture and quantify risks associated with nonregular or irregular market conditions. This is relevant in areas like option pricing, risk hedging, and portfolio optimization.

In robotics and motion planning, the $p$-regularity theory can  be used to analyze and plan the motion of robotic systems in complex environments. By incorporating higher-order derivatives, it becomes possible to model and optimize the robot's movements, ensuring stability and avoiding irregularities or singular configurations. This has applications in areas like robot control, autonomous navigation, and motion planning.

These examples illustrate the versatility and broad applicability of $p$-regularity theory in various fields of mathematics and applied sciences, where nonregular or degenerate problems arise.  The theory offers a valuable approach to analyzing and understanding nonlinear phenomena in situations where classical regularity assumptions fail, allowing for a more comprehensive and nuanced understanding of complex systems.
By incorporating higher-order derivatives and adapting the analysis to handle irregularities, $p$-regularity techniques provide valuable tools for understanding and characterizing complex phenomena in these settings.

\subsection{Aims and scope}
The main focus of this work is on analysis and solving nonlinear equations of the form
\begin{equation}\label{F(x)=0}
  F(x)=0,
\end{equation}
and optimization problems of the form
\begin{equation}\label{phi-min}
\min f(x) \quad \hbox{subject to } \; F(x)=0,
\end{equation}
where $f: X\rightarrow \mathbb{R}$ and $F: X\rightarrow Y$ are sufficiently smooth mappings, and $X$ and $Y$ are Banach
spaces. Many interesting applied nonlinear problems can be written in one of these forms. 

Nonlinear mappings $F$ and problems of the form \eqref{F(x)=0} and \eqref{phi-min} can be divided into two classes, called regular (or nonsingular) and singular (or degenerate). The classification depends on the mapping~$F$, which is either regular (that is, $F'(\bar{x}$) is onto) or singular (that is if $F'(\bar{x}$) is not onto). Roughly speaking, regular mappings are those for which implicit function theorem arguments can be applied and singular problems are those for which they cannot, at least, not directly.

The purpose of this paper is to give an overview of methods and tools of the $p$-regularity theory and to show how they can be applied to analyze and develop methods for solving singular (irregular, degenerate) nonlinear equations and
 equality-constrained optimization problems.
 The development of the theory of $p$-regularity started approximately in 1983--1984 with the concept of $p$-regularity introduced by Tret'yakov in \cite{Tr83, Tr84}.

One of the main results of the theory of  $p$-regularity gives a detailed description of the structure of the zero set 
of a nonregular nonlinear mapping $F: X \to Y$. 
It is interesting to note that there have been several examples in the history of mathematics when fundamental results were obtained independently in the same general time period. One such example related to the theory of $p$-regularity  concerns theorems about the structure of the zero sets of an irregular mapping satisfying a special higher--order regularity condition. The result that we are referring to was simultaneously obtained by Buchner, Marsden and Schecter \cite{BMS} and Tretyakov \cite{Tr84}. Approaches proposed in \cite{BMS} and in \cite{Tr84} are the same.
 The difference is in motivation and the context for the main result in both papers. In \cite{BMS}, the structure of the zero set around a point where the derivative is not surjective was studied in the context of the bifurcation theory. 
Theorem 1.3 in \cite{BMS} is referred to as a {\em blowing-up result}. 
In Fink and Rheinboldt \cite{FR:87}, it was noted that Theorem 1.3 in
\cite{BMS} was a powerful generalization of Morse Lemma and some interesting counterexamples for a naive approach to the Morse Lemma were found.  The same theorem derived by Tretyakov \cite{Tr84} is one of the main results for the $p$--regularity theory. 
The result led to various theoretical developments and applications of the theory to nonregular (or degenerate) problems in many areas of mathematics.
We should note that the results and constructions introduced by Marsden and Tret'yakov are the same in the completely degenerate case.

The paper is organized as follows. We discuss essential nonlinearity and singular mappings in Section~\ref{sec:2}.
 Then we recall the main concepts and definitions of the $p$-regularity theory in Section~\ref{sec4}.
 We discuss some classical results of analysis and methods for solving nonlinear problems via the $p$-regularity theory
 in Section~\ref{sec:Examples}. 
In each subsection, we focus on singular problems that illustrate that the classical results are not necessarily satisfied in the nonregular case. We present generalizations of the same classical results, which were derived during the last forty years using the constructions and definitions of the $p$-regularity theory.

 In this manuscript, we consider a variety of applications. 
 We start Section~\ref{sec:Examples}  with Lyusternik theorem in Section~\ref{Sec:LT}.
Lyusternik theorem plays an important role in the description of the solution sets of nonlinear equations and feasible sets of optimization problems in the regular case.
 However, the classical Lyusternik Theorem might not hold if
 mapping $F$ is singular at some point $\bar{x}$. 
The first generalization of the classical Lyusternik theorem for
$p$-re\-gu\-lar mappings
 was derived and proved simultaneously in \cite{BMS} and \cite{Tr83}. It can be applied to describe the zero set of a $p$-regular mapping.
 Representation Theorem and Morse Lemma are also presented in Section~\ref{Sec:LT}. 
We continue with consideration of the 
Implicit Function Theorem in Section~\ref{sec:IFT}.
 There are numerous books and papers devoted to the classical implicit function theorem, amongst which are \cite{Hamilton, Krantz}.
 The classical Implicit Function Theorem is not applicable in the case when a mapping $F : X \times Y \to Z$ is not regular, that is when $F_y^\prime (\bar{x}, \bar{y})$ is not onto for some $(\bar{x}, \bar{y})$. 
We present a generalization of the implicit function theorem for nonregular mappings. 
In Section~\ref{sec:Newton}, we cover the $p$-factor Newton's method for solving nonlinear equation \eqref{F(x)=0} and finding critical points of an unconstrained optimization problem.
Optimality conditions for equality-constrained optimization problems and Lagrange multiplier theorems for the regular and degenerate cases are considered in Section~\ref{sub3.2}.
 The modified Lagrange function method for   2-regular problems is covered in Section~\ref{sub3.6}.
 Singular problems of the calculus of variations and optimality conditions for $p$-regular problems of the calculus of variations are considered in 
Section~\ref{sub3.5}.
 The existence of solutions to nonlinear equations in regular and degenerate cases is covered in 
Section~\ref{sec:Ex}.
 The second-order nonlinear ordinary differential equations with boundary conditions are presented in Section~\ref{sub3.9}.
 Newton interpolation polynomials and the $p$-factor interpolation method are considered in Section~\ref {sub4.11.1}.
  We make some concluding remarks in Section~\ref{sec:C}.

\paragraph{\bf General Notation}  Let
$\mathcal{L} (X,\, Y)$ be the space of all continuous linear ope\-ra\-tors
from $X$ to $Y$ and for a given linear operator $\Lambda
:X\rightarrow Y$,  let us denote its kernel and image by $\Ker \Lambda
=\{ x\in X\mid
\Lambda x=0\}$ and $\operatorname{Im} \Lambda =\{y\in Y\mid  y=\Lambda
x\mbox{ for some } x\in X\}$, respectively.
Also, $\Lambda ^*:Y^*\to X^*$ denotes the
adjoint of $\Lambda $, where $X^*$ and $Y^*$ denote the dual spaces of $X$
and $Y$, respectively.

Let $p$ be a natural number and let $B : X \times X \times \ldots
\times X \; \mbox{(with $p$ copies of $X$)} \;\to Y$ be a continuous
symmetric
$p$-multilinear mapping.  The $p$-form {\em associated to} $B$ is the map
$B[ \cdot ] ^p: X
\rightarrow Y$  defined by
$$
     B[ x ] ^p = B( x, x, \ldots, x),
$$
for $x \in X$. Alternatively,
we may simply view $B[ \cdot ] ^p$ as a homogeneous polynomial $Q: X
\rightarrow Y$ of degree~$p$,  i.e., $Q(\alpha x) = \alpha^p Q (x)$. The
space of continuous homogeneous polynomials $Q: X
\rightarrow Y$ of degree $p$ will be denoted by $\mathcal{Q}^p(X, \, Y)$.

If $F: X \rightarrow Y $ is of class $C^2$, its derivative $F^\prime$ at a point $x\in X $ is a linear continuous operator, $F ^\prime (x)\in\mathcal{L} (X,\, Y) $, i.e. $F':X\rightarrow \mathcal{L} (X,\, Y) $ and $F^{\prime \prime}:X\rightarrow \mathcal{L} (X,\, \mathcal{L} (X,\, Y))$. Hence, $F^{\prime \prime}(x)\in \mathcal{L} (X,\, \mathcal{L} (X,\, Y))$,  $F^{\prime \prime}(x)(h_{1})\in \mathcal{L} (X,\, Y)$, $F^{\prime \prime}(x)(h_{1}, h_{2})\in Y$, $h_{1},h_{2}\in X$, and the mapping $(h_{1},h_{2})\rightarrow F^{\prime \prime}(x)(h_{1}, h_{2})$ is a continuous symmetric bilinear mapping, see e.g.  \cite{MR0349288}, Chapter Viii.

If $F: X \to Y$ is of class $C^p$, we let $F^{(p)}(x)$ be the $p$th-order
derivative of
$F$ at the point $x$ (a symmetric multilinear map of $p$ copies of $X$ to
$Y$) and the associated
$p$-form, also called the {\em $p$th--order mapping}, is
$$
    F^{(p)}(x) [ h]^p = F^{(p)}(x) ( h, h, \ldots, h).
$$
Furthermore, we use the following key notation for the {\em $p$-kernel} of the $p$th-order mapping:
$$
     \Ker^{p} F\, ^{(p)} (x) = \{ h\in X \, |\, F\,
^{(p)}(x) \, [h]^p = 0 \, \} .
$$
This set is also called the {\em locus} of $F^{(p)}(x)$.

\section{Essential nonlinearity and singular mappings}
\label{sec:2}

Let a mapping $F : X \rightarrow Y$ belong to  $\mathcal{C}^1(W)$, where $W$ is a~neighborhood of some point $\bar{x} \in X$.
 According to \mbox{Definition \ref{regular_point}}, a mapping $F$ is called \textit{regular} at $\bar{x}$, if
\begin{equation}\label{2eq1}
 \operatorname{Im} F'(\bar{x})=Y.
\end{equation}

The following lemma on the local representation of a regular mapping holds.

\begin{lemma}[Lemma 1., Sec.1.3.3. of \cite{IzTr94}]
\label{2lem1}
Let $X$ and $Y$ be Banach spaces, $W$ be a~neighborhood of some point $\bar{x} \in X$, and $F : X \rightarrow Y$ be $\mathcal{C}^1(W)$.
If $F$ is regular at $\bar{x}$, then  there exist a neighborhood $U$ of $0$, a neighborhood $V$ of
	$\bar{x}$, and a diffeomorphism $\varphi:U\rightarrow V$ such that
	\begin{enumerate}
		\item $\varphi(0)=\bar{x}$,
		\item $F(\varphi(x))=F(\bar{x})+F'(\bar{x})x$ for all $x\in U$,
		\item  $\varphi'(0) = I_X$ (the identity mapping on $X$).
		\end{enumerate}
\end{lemma}

\ref{2lem1} says that  the diffeomorphism $\varphi$  locally transforms $F$ into the linear
mapping:
\begin{equation}
\label{2eq2}
F(\varphi(x))=F(\bar{x})+F'(\bar{x})x \ \ \text{ for all } x\in U.
\end{equation}
This fact is also  referred to as the local ``trivialization theorem'' (Theorem~1.26 of \cite{ioffe}).
If the regularity condition \eqref{2eq1} is not satisfied, then, in general, the local linearization of $F$ is not possible ($\varphi$ does not exist).

There exist numerous mappings which do not admit local linearization.
 The concept of essentially nonlinear mappings defined in \cite{TrMa03} formalizes this situation.

\begin{definition}
\label{2def2}
Let $V$ be a neighborhood of $\bar{x}$ in $X$ and $U\subset X$ be a neighborhood of $0$. A
mapping $F:V\rightarrow Y$, $F\in  \mathcal{C}^2(V)$, is
 \textit{essentially nonlinear at } $\bar{x}$ if there exists a~perturbation of the form
$$\widetilde{F}(\bar{x}+x)=F(\bar{x}+x)+\omega(x), \hbox{ where } \|\omega(x)\|=o(\|x\|),$$
such that there does not exist any nondegenerate transformation
  $\varphi(x):U\rightarrow V$,  $\varphi\in \mathcal{C}^1(U)$,  such that  $\varphi (0) = \bar{x}$,
   $\varphi ' (0) = I_X$ and Equation \eqref{2eq2}  holds with
$\varphi$ and $\widetilde{F}$.
\end{definition}

\begin{definition}
\label{2def3}
We say that mapping $F$ is \textit{singular} (or  \textit{degenerate})
at $\bar{x}$ if it fails to be regular; that is, its derivative is not onto:
\begin{equation}
\label{2eq3}
\operatorname{Im}  F'(\bar{x})\neq Y.
\end{equation}
\end{definition}

The following Theorem \ref{l:es}, which establishes the relationship between these
two notions of the essential nonlinearity and singularity, was derived in \cite{TrMa03}.  We give its proof here for completeness of our developments.

\begin{theorem}
\label{l:es}
    Suppose $F:V \to Y$ is $C ^2$ and  $\bar{x}$ is a solution of
\textup{(\ref{F(x)=0})}. Then
$F$ is essentially nonlinear
    at the point $\bar{x}$ if and only if $F$ is singular
    at the point $\bar{x}$.
\end{theorem}

\begin{proof}
Suppose that $F$ is singular at the point $\bar{x}$, $F(\bar{x})=0$, i.e.,
$\operatorname{Im} \, F'(\bar{x}) \neq Y$, so there exists a nonzero
element
$\xi \in Y$ such that  $\| \xi \| = 1$ and
\begin{equation} \label{xi}
   \xi \notin \operatorname{Im} \, F'(\bar{x}) .
\end{equation}
Assume on the contrary that $F$ is not essentially nonlinear at $\bar{x}$.
Define mapping $\tilde{F}$ as
\begin{equation} 
\label{tildeFf}
      \tilde{F} (\bar{x} + x) =  F(\bar{x}) + F'(\bar{x}) x + \xi \| x \|^2.
\end{equation}
Note that $ \xi \| x \|^2 \  \notin \operatorname{Im} \, F'(\bar{x}) $  for any $x\in V$.

By virtue of the above assumptions,
\ref{2lem1} and Equation (\rm \ref{tildeFf}), there exist a neighborhood $U$ of~$0$ and a  mapping $\varphi(x):U\rightarrow V$, $\varphi\in \mathcal{C}^1(U)$, such that $\varphi(0)=\bar{x}$, $\varphi'(0)=I_X$ and
\begin{equation} \label{rep2}
    \tilde{F}(\varphi (x))=\tilde{F}(\bar{x} )+  \tilde{F}' (\bar{x} ) x
                          = F(\bar{x} )+  F' (\bar{x} ) x
               \end{equation}
for all $x\in U $. Since $ F(\bar{x} )=0$ and $ F' (\bar{x} ) x  \in
\operatorname{Im}  F' (\bar{x} ) $, then from (\rm \ref{rep2}) we have
\begin{equation} \label{conclus1}
     \tilde{F}(\varphi (x)) \in  \operatorname{Im}  F' (\bar{x} ).
\end{equation}
However,
using $F(\bar{x} )=0$, $\varphi (0) = \bar{x}$ and $\varphi ' (0) =
I_X$,  we obtain
\begin{equation} \label{F(phi)}
  \begin{array} {ll}
    \tilde{F}(\varphi (x)) & = F(\bar{x} + (\varphi (x) - \bar{x})) \\[2mm]
     & =
        F(\bar{x}) + F'(\bar{x})(\varphi(x) - \bar{x}) + \xi \| \varphi(x) - \bar{x} \|^2 \\[2mm]
     & =
        F'(\bar{x})(\varphi(x) - \bar{x}) +
         \xi \| \varphi(0) + \varphi'(0)x+ \omega_1(x) -  \bar{x} \|^2 \\[2mm]
     & =   F'(\bar{x})(\varphi(x) - \bar{x}) +
         \xi \|x+ \omega_1(x)\|^2 ,
   \end{array}
\end{equation}
where $\| \omega_1(x)\| = o (\|x\|)$. Thus, for small $x$,
$$
     \xi \|x+ \omega_1(x)\|^2  \neq 0.
$$
Taking into account (\rm \ref{xi}),  (\rm \ref{F(phi)}) and the fact that
$ F' (\bar{x} ) (\varphi(x)-\bar{x})  \in \operatorname{Im} F' (\bar{x}
) $, we conclude from this that
\begin{equation} \label{conclus2}
     \tilde{F}(\varphi (x)) \notin  \operatorname{Im}  F' (\bar{x} ).
\end{equation}
This contradicts \eqref{conclus1} and therefore $F$ is essentially
nonlinear at $\bar{x}$.

\medskip

To prove the converse, suppose that $F$ is essentially nonlinear at
$\bar{x}$,  but that $F$ is not singular; i.e., is regular at this point.
Then by persistence of the regularity condition, for any perturbation
$$ \widetilde{F}(\bar{x}+x) = F(\bar{x}+x) + \omega(x), $$
where $ \| \omega(x) \| = o (\|x \|)$,
the map $ \widetilde{F}(\bar{x}+x)$ is regular at $\bar{x}$ and
$ F'(\bar{x}) = \widetilde{F}'(\bar{x})$.
Hence, by virtue of a Theorem concerning the representation of a regular
mapping (see Izmailov and  Tret'yakov
\cite{IzTr94})\footnote{Under additional splitting assumptions, which are
not made here, this result would be a standard consequence of the implicit
function theorem, as in, for example, \cite{AbMaRa1988}, \S2.5.},
$ \widetilde{F}(\bar{x}+x)$ is represented as
$$
    \tilde{F}(\varphi (x))  = \tilde{F}(\bar{x}) +  \tilde{F}'(\bar{x}) x,
$$
where $\varphi (0) = \bar{x}$ and $\varphi ' (0) = I_X$.
It contradicts to the definition of essential
nonlinearity of the mapping~$F$.
\end{proof}

\section{Elements of $p$-regularity theory}
\label{sec4}
For the purpose of describing essentially nonlinear problems,
the concept of $p$-regularity was introduced by Tret'yakov \cite{Tr83,
Tr84,  Tr87} using the notion of a $p$-factor  operator. Let us recall main definitions of the $p$-regularity theory, which are presented, for example, in \cite{IzTr94, Tr87, TrMa03}.

We construct the  $p$-factor  operator  under the {\it assumption} that the
space $Y$ is decomposed into the (topological) direct sum
\begin{equation} \label{Y=sum}
   Y = Y_1 \oplus \ldots \oplus Y_p,
\end{equation}
where
$
   Y_1 = {\rm cl}\, (\operatorname{Im} \, F '(\bar{x})),
$
the closure of the image of the first derivative of $F$ evaluated at 
$\bar{x}$, and the remaining spaces are defined as follows.  Let $Z_{1}=Y$ and let
$Z _2 $ be a closed complementary subspace to $Y_1 $ (we are
{\it assuming} that such a closed complementary subspace exists) and
let $P_{Z_2} : Y \to Z_2$
be the projection operator onto $Z_2$ along $Y_1$. Let
$Y_2$ be the closed linear span of the image of
the quadratic map $P_{Z_2}F ^{(2)}(\bar{x})[ \cdot ] ^2$. More
generally, define inductively,
$$
   Y_i = {\rm cl}\, ({\rm span}
\; \operatorname{Im} P_{Z_i} F^{(i)} (\bar{x}) [\cdot]^i )
   \subseteq Z_i,\quad
   i = 2,\,\dots, p-1, \ \  p>2,
$$
where $Z_i$ is a choice of the closed complementary subspace for $(Y_1 \oplus
\ldots \oplus Y_{i-1})$ with respect to $Y$, $i=2,\,\dots,p$, and $P_{Z_i}
: Y \to Z_i$ is the projection operator onto $Z_i$ along $(Y_1 \oplus
\ldots \oplus Y_{i-1})$ with respect to $Y$, $i=2,\,\dots,p$.
Finally, let
$
   Y_p = Z_p.
$
The order $p$ is chosen as the minimum number for which
Equation \eqref{Y=sum} holds. In particular, for $p=2$, we have $Y=Z_{1}=Y_{1}\oplus Z_{2}$. 
When $Y$ is a Hilbert space, there exists a complementary subspace  to $Y_{1}$,  i.e. the orthogonal subspace, $Y_{2}=Y_{1}^{\perp}$.

\begin{remark}
    The subspaces $Y_i$ in assumption \eqref{Y=sum} can be replaced in further consideration by subspaces constructed using the so-called factorization procedure. Namely,   $$ Y_1 = {\rm cl}\, (\operatorname{Im} \, F '(\bar{x})),$$ as before, but, instead of  $Y_2$, 
    the space $ Y/Y_1$, called {\em a quotient or factor space}, is used. 
    Note that the quotient space is a Banach space, see {\em e.g.} \cite{optcontrol}. Moreover, 
    if \eqref{Y=sum} holds, then $Y_2$ is isomorphic to  $ Y/Y_1$. We use the assumption \eqref{Y=sum} for simplicity of our presentation.
\end{remark}

Define the following mappings (see Tret'yakov \cite{Tr87})
\begin{equation}
\label{4eq2}
   f_i(x): X \to Y_i, \quad
   f_i(x) = P_{Y_i} F(x), \quad i=1,\,\dots,p,
\end{equation}
where
$P_{Y_i} : Y \to Y_i$ is the projection operator onto
$Y_i$ along
$(Y_1 \oplus \ldots \oplus Y_{i-1} \oplus Y_{i+1} \oplus \ldots \oplus Y_p)$
with respect to $Y$, $i=1,\,\dots,p$.
Recall that  $P_{Y_{i}}$ is the projection onto
$Y_i$ along (or parallel to)
$W_{i}=(Y_1 \oplus \ldots \oplus Y_{i-1} \oplus Y_{i+1} \oplus \ldots \oplus Y_p)$ if $\Ker  P_{Y_{i}}=W_{i}$. Moreover, 
\begin{equation} 
\label{image_projection}
f^{(k)}_{i}(\bar{x})=P_{Y_{i}}F^{(k)}(\bar{x})=0, \quad i=1,\ldots,p, \quad k=1,\ldots,i-1.
\end{equation}

\begin{definition}
\label{4def1}
    The linear operator $\Psi_p(h) \in \mathcal{L} (X, Y_1 \oplus \ldots
\oplus Y_p)$, where $ h \in X$,  $h\neq 0$, is defined by
\begin{equation}
\label{4eq3}
      \Psi_p(h) = f_1 '(\bar{x})  +  f_2 ''(\bar{x}) [h]  + \ldots +
                       f_p^{(p)} (\bar{x})[h]^{p-1} ,
\end{equation}
and is called the {\em $p$-factor  operator}. Alternatively, the following form of the  $p$-factor  operator can be used:
$$   \Psi_p(h) = f_1 '(\bar{x})  + \frac{1}{2!} f_2 ''(\bar{x}) [h]  + \ldots +
                      \frac{1}{p!} f_p^{(p)} (\bar{x})[h]^{p-1} .
$$
\end{definition}

\begin{remark} 
Every mapping $f_{i}(x)$ is completely degenerate at the point $\bar{x}$ up to the order $i-1$, $i=1,\ldots,p$. 
\end{remark}

 Note that in the {\em completely degenerate case}, i.e., in the case when
$$
   F^{(r)}(\bar{x}) = 0, \quad r=1,\dots,p-1 ,
$$
the $p$-factor  operator  is simply $F^{(p)} (\bar{x})[h]^{p-1}$. Observe that, when $Y_{1}=Y$, i.e. $F$ is regular at $\bar{x}$, then $\Psi_p(h)\equiv F'(\bar{x})\equiv f_{1}'(\bar{x})$.

For $p=2$, the $p$-factor-operator \eqref{4eq3} takes the form
\begin{equation}
\label{4eq4}
	\Psi_{2}(h)=f'_{1}(\bar{x})+
	f''_{2}(\bar{x})h,
\end{equation}
or, equivalently,
$\Psi_{2}(h)=f'_{1}(\bar{x})+
	\dfrac12 f''_{2}(\bar{x})h,
	$
	where $h\in X$, $h \neq 0$. In view of \eqref{image_projection} we see that the construction of the  operator $\Psi_{2}(h)$ (and $\Psi_{p}(h)$, in general) is strongly related to the decomposition of the image space \eqref{Y=sum}. The idea is to use higher-order derivatives of $F$ to obtain \eqref{Y=sum}, whenever possible. In particular,  for $p=2$, and $\Psi_{2}(h)$ given by \eqref{4eq4}, we seek for those $h\in X$ that ensure the equality $\operatorname{Im} f''_{2}(\bar{x})h=Z_{2}$, where $Z_{2}$ is the complementary space to $Y_{1}$.

Recall that a bounded linear operator $T: X \to Y$ between Banach spaces $X$ and $Y$ is Fredholm if  the kernel of $T$ has a finite dimension and the image of $T$ is a closed subspace of a finite
		codimension in Y. 
Hence, in the case of Fredholm operator $F'(\bar{x})$, the subspace $Y_{1}=\operatorname{Im} F'(\bar{x})$ has a complementary finite-dimensional subspace $Z_2$ such that $Y=Y_{1}\oplus Z_{2}$.

\begin{definition}
\label{4def3}
 We say that the mapping $F$ is \emph{$p$-regular at $\bar{x}$} along an element $h \in X$, if
$$\hbox{Im}\ \Psi_{p}(h)=Y.$$
\end{definition}

\begin{remark}
\label{4rem4}
The condition of $p$-regularity of the mapping $F$ at the point $\bar{x}$ along $h \in X$ is equivalent to the following condition
\begin{equation}
\label{eq-rem-3.7}
\operatorname{Im}  f^{(p)}_p(\bar{x})[h]^{p-1}\left( \Ker   \Psi_{p-1}(h)\right)=Y_p,
\end{equation}
where $\Psi_{p-1}(h)=f_1'(\bar{x})+f_2''(\bar{x})[h]+\cdots +f_{p-1}^{(p-1)}(\bar{x})[h]^{p-2}.$ 
In particular, when $p=2$, $\Psi_1(h)=f_1'(\bar{x})$ and \eqref{eq-rem-3.7} 
takes the form $\operatorname{Im} f^{\prime \prime}_2(\bar{x})[h]\left( \Ker   f_{1}'(\bar{x})\right)=Y_2$, which is a consequence of elementary algebraic facts.

\end{remark}

We also define the $k$-kernel of the $k$th-order mapping $f_k^{(k)}(\bar{x})$ by
\begin{equation}
\label{KKer}
    \Ker^k f_k^{(k)}(\bar{x})=\{ \xi \in
    X \mid f_k^{(k)}(\bar{x})[\xi]^k=0\}.
\end{equation}

\begin{definition}\label{definition_H_p}
We say the mapping $F$ is {$p$-regular at  $\bar{x}$}
 if it is $p$-regular along any $h$ from the set
\begin{equation} \label{Hp}
   H_p(\bar{x}) =
    \left\{h \in X \mid  h \in \mathop{\bigcap}\limits_{i=1}^p \operatorname{Ker}^i f_i^{(i)}
(\bar{x})
\right\}
    \backslash \{ 0 \},
\end{equation}
where the $i$-kernel of the $i$th-order mapping $f_i^{(i)}(\bar{x})$ is
defined in Equation \eqref{KKer}.
\end{definition}

For a linear surjective operator  $\; \Psi_{p}(h) :X\rightarrow Y$
between Banach spaces, we denote its
\emph{right inverse} by $\; \{\Psi_{p}(h)\}^{-1}$ (see \cite{theoryofetremal}).
Therefore, $\{\Psi_{p}(h)\}^{-1}:Y
\rightarrow 2^{X}$ and we have
\begin{equation}
\label{4eq5}
\{\Psi_{p}(h)\}^{-1}(y)=\left\{x\in X\mid \Psi_{p}(h) x =
y\right\}.
\end{equation}
We define the \textit{norm} of $\{\Psi_{p}(h)\}^{-1}$ by
\begin{equation}
\label{4eq6}
\|\{\Psi_{p}(h)\}^{-1}\|=\sup_{\|y\|=1}\inf \{\|x\| \mid 
x\in\{\Psi_{p}(h)\}^{-1}(y)\}.
\end{equation}
We say that $\{\Psi_{p}(h)\}^{-1}$ is \emph{bounded}  if
$\|\{\Psi_{p}(h)\}^{-1}\|<\infty.$

\begin{definition}
\label{5def2}
A mapping $F\in \mathcal{C}^p$ is called {strongly $p$-regular} at a point $\bar{x}$ if there exists $\alpha >0$ such that
$$\sup_{h\in H_{\alpha}} \left\|\{\Psi_p(h)\}^{-1}\right\|<\infty,$$
where $\{\Psi_p(h)\}^{-1}$ is the right inverse operator  of $\Psi_p(h)$ and 
 $$
         H_{\alpha} = \left\{ h \in X \, | \, \|f_i^{(i)} (\bar{x}) [h]^i
\|_{Y_i}
\leq
         \alpha \quad \mbox{{\rm for all}} \quad  i =1,\ldots,p, \quad \|h
\|_X = 1
\right\}.
     $$
\end{definition}



\begin{example}
 \label{3ex1}
Consider mapping $F:\mathbb{R}^{2}\rightarrow \mathbb{R}^{2}$ defined by
    $$F(x)=\left(
           \begin{array}{c}
             x_1+x_2 \\
             x_1x_2 \\
           \end{array}
         \right).
$$
Let $\bar{x}=(0,0)^T$. Then the Jacobian $
    F'(\bar{x})=\left(
           \begin{array}{cc}
             1 & 1 \\
             0& 0 \\
           \end{array}
         \right)
$ is singular (degenerate) at $\bar{x}.$
Hence,
$\operatorname{Im}  F'(\bar{x})=\operatorname{span}\{(1,0)\}\neq \mathbb{R}^2$. Let
   $Y_1=\operatorname{span}\{(1,0)\}$ and
 $Y_2=\operatorname{span} \{ (0,1)\}.$
 To construct the $2$-factor operator, we use the projection matrices
 $$
P_{Y_1}=\left(
 \begin{array}{cc}
 	1 & 0 \\
 	0 & 0 \\
 \end{array}
 \right) \quad  \mbox{and} \quad P_{Y_2}=\left(
 \begin{array}{cc}
 0 & 0 \\
 0 & 1 \\
 \end{array}
 \right).
$$
According to Equation \eqref{4eq2}, the mappings $f_1:\mathbb{R}^2\rightarrow Y_1$ and $f_2:\mathbb{R}^2\rightarrow Y_2$ have the form
 $$
 f_1(x)=\left(
           \begin{array}{c}
             x_1+x_2 \\
             0 \\
           \end{array}
         \right) \quad   \mbox{and} \quad    f_2(x)=\left(
           \begin{array}{c}
             0 \\
             x_1x_2 \\
           \end{array}
         \right). 
$$
 Then 
$$
 f_1^{\prime}(x)=\left(
           \begin{array}{cc}
             1 & 1 \\
             0 &0\\
           \end{array}
         \right),\quad      f_2^{\prime}(x)=\left(
           \begin{array}{cc}
             0& 0 \\
             x_2 & x_{1} \\
           \end{array}
         \right) 
$$
and 
$$
 f_2^{\prime \prime}(x)h=\left(
           \begin{array}{cc}
             0& 0 \\
             h_{2} & h_{1} \\
           \end{array}
         \right).
$$

Hence, for $h=(h_1,h_2)^T \in \mathbb{R}^2$, the $2$-factor operator is defined by
$$
\Psi_{2}(h)=f'_{1}(\bar{x})+
	f''_{2}(\bar{x})h =
\left(
           \begin{array}{cc}
             1& 1 \\
             h_{2} & h_{1} \\
           \end{array}
         \right).
$$
It is easy to see that if $h_1\neq h_2$, then the $2$-factor operator is surjective.

In this example, we have
$$\Ker^1f'_1(\bar{x})=\operatorname{span} \left\{ \left(1,-1\right)\right\} \quad \mbox{and}
\quad  \Ker^2f''_2(\bar{x})=\operatorname{span} \left\{ \left( 1,0 \right)\right\}\cup \operatorname{span} \left\{ \left(0,1 \right)\right\}.$$
It means that $H_{2}(\bar{x})=\emptyset$. 
Hence, according to \ref{4def3}, mapping $F$ is 2-regular at $\bar{x}$ along any $h\in X$ such that $h_1\neq h_2$, but is not 2-regular at $\bar{x}$. 
As we see, it may happen that $F$ is $2$-regular along some $h\in X$ but $H_{p}(\bar{x})=\emptyset$.
Therefore, a given mapping $F$ may be not 2-regular with respect to
all $h\in X$, $h \neq 0$.
\end{example}

\begin{example}
\label{ex:3-regularity}

Case $p=3$. Consider mapping $F:\mathbb{R}^{2}\rightarrow \mathbb{R}^{3}$ defined by
    $$F(x)=\left(
           \begin{array}{c}
             x_1+x_2 \\
             x_1x_2^{2} \\
             x_{1}^{3}
           \end{array}
         \right).
$$
With $\bar{x}=(0, 0)^T = 0$ we get
    $$F'(x)=\left(
           \begin{array}{cc}
             1&1 \\
             x_2^{2}&2x_{1}x_{2} \\
             3x_{1}^{2}&0
           \end{array}
         \right)
 \quad \mbox{and} \quad F'(0)=\left(
           \begin{array}{cc}
             1&1 \\
             0&0\\
             0&0
           \end{array}
         \right).
$$
                  Then with $h = (h_1, h_2)^T$,
         $$
F^{\prime}(x)h=\left( \begin{array}{cc}
             h_1 + h_2 \\
            x_2^{2}h_1 + 2x_{1}x_{2} h_2 \\
             3x_{1}^{2}h_1
           \end{array}
         \right),
         $$
       $$
F^{\prime \prime}(x)[h]^2=\left( \begin{array}{c}
             0 \\
            4x_{2}h_1 h_2 + 2x_{1} h_2^2 \\
            6x_1 h_1^2
           \end{array} \right), \quad 
F^{\prime \prime \prime}(x)[h]^2=
         \left( \begin{array}{cc}
             0 & 0\\ 
            2h_{2}^{2} & 4h_{1}h_{2}\\
            6h_{1}^{2} & 0
           \end{array}
         \right).
         $$
  In this example, 
$$Y_{1}=\text{Im} F'(0)={\rm span}\{(1,0,0)\}, \quad Y_2 =(0,0,0), \quad Y_{3}={\rm span}\{(0,1,0),(0,0,1)\}.$$ To construct the $3$-factor operator, we use the projection matrices
$$
 P_{Y_1}=\left(
 \begin{array}{ccc}
 	1 & 0 & 0\\
 	0 & 0 & 0 \\
	 	0 & 0 & 0 
 \end{array}
 \right) \quad \mbox{and} \quad P_{Y_3}=\left(
 \begin{array}{ccc}
  	0 & 0 & 0 \\
 0 & 1 & 0 \\
  0 & 0&  1 
 \end{array}
 \right).$$
Then, by using Equation \eqref{4eq2}, we define
$f_{1}$ and $f_{3}$ as
$$
f_{1}(x)=
P_{Y_1} F(x)
=\left( \begin{array}{cc}
             x_{1}+x_{2} \\
             0\\
             0
           \end{array}
         \right) \quad \mbox{and} \quad
         f_{3}(x)=P_{Y_3} F(x)
=\left( \begin{array}{cc}
             0 \\
            x_1x_2^{2} \\
             x_{1}^{3}
           \end{array}
         \right).
         $$

By the definition of the 3-factor-operator, we get
\begin{eqnarray*}
\Psi_{3}(h) &=& f'_{1}(\bar{x})+
	f''_{3}(\bar{x})[h]^2,                \\ 
&=& 
\left( \begin{array}{cc}
             1&1 \\
             0&0\\
             0&0
           \end{array}
         \right)+ \left( \begin{array}{cc}
             0 & 0\\ 
            2h_{2}^{2} & 4h_{1}h_{2}\\
            6h_{1}^{2} & 0
           \end{array}
         \right)=\left(\begin{array}{cc}
             1 & 1\\ 
           2 h_{2}^{2} & 4 h_{1}h_{2}\\
            6h_{1}^{2} & 0.
           \end{array}
         \right).
\end{eqnarray*}

For $h=(h_{1},0)$, the $3$-factor operator takes the form
$$
\Psi_{3}(h)=\left(\begin{array}{cc}
             1 & 1\\ 
           0 & 0\\
            6h_{1}^{2} & 0
           \end{array}
         \right)
$$
and $\text{cl Im}\Psi_{3}(h)=\text{span}\{(1,0,0),(1,0,1)\}$.

For $h=(0, h_{2})$, the $3$-factor operator takes the form
$$
\Psi_{3}(h)=\left(\begin{array}{cc}
             1 & 1\\ 
           2h_{2}^{2} & 0\\
            0 & 0
           \end{array}
         \right)
$$
and $\text{cl Im}\Psi_{3}(h)=\text{span}\{(1,1,0),(1,0,0)\}$.

Now we calculate the kernels according to \eqref{KKer} with $\xi = (\xi_1, \xi_2)$:
$$
f_1^{\prime}(\bar{x})\xi=\left( \begin{array}{cc}
             \xi_1 + \xi_2 \\
0\\
  0
           \end{array}
         \right)= \left( \begin{array}{c}
         0\\0\\0 \end{array}\right), \quad 
         f^{\prime \prime \prime}_{3}(\bar{x})[\xi]^3 = \left( \begin{array}{c}
             0 \\ 
            6 \xi_{1} \xi^2_{2}\\
            6 \xi_{1}^{3} 
           \end{array}
         \right)
         = \left( \begin{array}{c}
         0\\0\\0 \end{array}\right)        
$$
so $$
         \Ker f_{1}^{\prime}(0)=\{\xi=(\xi_{1},\xi_{2})\ |\ \xi_{1}+\xi_{2}=0\}=\text{span}\{(1,-1)\}
         $$
         and 
         $$
           \Ker f_{3}^{\prime \prime \prime}(0)=\{\xi=(\xi_{1},\xi_{2})\ |\ 6\xi_{1}\xi_{2}^{2}=0,\   6\xi_{1}^{3}=0\}=\text{span}\{(0,1)\}.\ \
           $$
As easy to verify,           $$
           \Ker f_{2}^{\prime \prime}(0,0)=\mathbb{R}^{2}. $$

\end{example}

\medskip

\section{Singular problems and classical results via the $p$-regularity theory}
\label{sec:Examples}
\subsection{Lyusternik theorem and description of solution sets}
\label{Sec:LT}
Lyusternik theorem plays an important role in the description of the solution sets of nonlinear equations and feasible sets of optimization problems in the regular case.

The Lyusternik theorem has practical applications in various fields.
The theorem is important for study of optimization and variational problems. By describing the tangent cone, the theorem provides information about the critical points and the behavior of the solutions near those points.
In control theory, the Lyusternik theorem can be employed to analyze the stability and controllability of nonlinear control systems. By investigating the tangent cone, one can gain insights into the behavior of the system near critical points and determine the conditions for stability and controllability.
The Lyusternik theorem can also be utilized in the development and analysis of optimization algorithms, such as gradient-based methods. By characterizing the tangent cone, the theorem helps in designing efficient algorithms and understanding their convergence properties.

These are just a few examples of the practical applications of the Lyusternik theorem. Its insights into the tangent cone are valuable in various fields, ranging from optimization and control theory to PDEs and geometry, providing a deeper understanding of the behavior of solutions and critical points in different mathematical problems.

Let $X,Y$ be normed linear spaces. Consider a nonlinear mapping $F: U \to Y$, where $U$ is some neighborhood of a point $\bar{x} \in X$.
We are interested in the description of the {\em solution set} $M(\bar{x})$,
\begin{equation} \label{M}
M(\bar{x})= \left\{ x \in U \mid F(x)=F(\bar{x})\right\},
\end{equation}
for nonlinear equation
\eqref{F(x)=0}. This is also the feasible set for an
optimization problem \eqref{phi-min}. In particular, if
$\bar{x}$ is a  zero of $F$, the set $M(\bar{x})$ is the zero-set of $F$.

It will be useful to recall the following definition
of tangent vectors and tangent cones (see, for instance,  \cite{Clarke1983}).

\begin{definition}
\label{Def:TC}
    We call $h$ a {tangent vector} to a set $M \subseteq X$ at $\bar{x}
\in M$
    if there exist $\varepsilon > 0$ and
    a function $r: [0,\varepsilon] \rightarrow X$ with the property that
    for $t \in [0,\varepsilon]$, we have
    $
       \bar{x}+th+ r(t) \in M
    $
    and
    $$
       \mathop{\rm lim}\limits_{t \to 0}  \dfrac{\|r(t)\|}{ t} =  0.
    $$
    The collection of all tangent vectors at $\bar{x}$ is called the
{tangent cone}
    to $M$ at $\bar{x}$ and it is denoted by  $T_1 M(\bar{x})$.
\end{definition}

\subsubsection{Lyusternik theorem in the regular case}
  In the regular case, the Lyusternik theorem(see \cite{LiSo61}) can be formulated as follows.

\begin{theorem}[Lyusternik Theorem]
\label{Th:LT}
    Let $X$ and $Y$ be Banach spaces and $U$ be a neighborhood
    of $\bar{x}$ in $X$.
    Suppose $F: U\to Y$ is Fr\'echet differentiable on $V$, and the mapping
    $F': U\to \mathcal{L}(X,\, Y)$ is continuous at $\bar{x} $.
    Suppose further that $F$ is re\-gu\-lar at $\bar{x}$.

    Then  the tangent cone to the set
    $M(\bar{x})$ defined in \eqref{M}
    is the linear space that is equal to the kernel of $F'(\bar{x} )$:
    \begin{equation} \label{Ker=M}
      T_1 M(\bar{x}) = \Ker  F'(\bar{x} ).
    \end{equation}
\end{theorem}

  If
  $F$ is singular at  $\bar{x}$,
  then, in some problems, $T_1M(\bar{x})\neq  \Ker  F'(\bar{x})$, as in the following example.

\begin{example}
\label{Ex:L}
Let $X = \mathbb{R}^2$, $x = (x _1, x _2) \in X$, and $F : \mathbb{R}^2 \rightarrow \mathbb{R}$ be defined by $F(x) =x_1^2-x_2^2+o(\|x\|^2).$
For the point
$\bar{x}=0$, we get $F(0)  = 0$, $\Ker  F'(0)= \mathbb{R}^2$, and  $T_1M(0)={\rm span} \{(1,1)\} \cup {\rm span} \{(1,-1)\}$. Hence, $T_1M(0) \neq \Ker  F'(0)$.
\end{example}

\begin{example}
Let $F: C([0,1])\rightarrow C([0,1])$ be defined as $F(x(\cdot))= x(\cdot).$ Then $M=\{ x(\cdot)\in C([0,1]) \mid F(x(\cdot))=F(0)=0 \}= \{ 0\}$ and $F'(x(\cdot))=1$ is onto. We have $\Ker F'(0)=\{0\}=T_1M(0).$ 
\end{example}

\begin{example}
    Let $M=\left\{ x(\cdot)\in C([0,1])\mid \int\limits_{0}^{1} \sin x(t) dt = \dfrac 2 \pi \right\}$ and $\bar{x}=\pi t$. To calculate $T_{1} M(\bar{x})$, it is enough to apply Lyusternik theorem with $F:C([0,1]) \rightarrow \mathbb{R}$ is defined as $F(x(\cdot))=\int\limits_0^1 \sin x(t) dt.$ 
    
    
    We have $$\frac{F(x+h)-F(x)}{\|h\|}= \int_0^1 \sin x(t) \frac{\cos h(t)-1}{\|h\|} dt + \int_0^1 \cos x(t) \frac{\sin h(t)}{\|h\|} dt .$$ 
    First term in the right-hand side approaches $0$ as $\|h\|$ goes to $0$. In the second term, we can use the fact that $\dfrac{\sin h(\cdot)}{\|h\|}$ approaches $1$ as $\|h\|$ goes to 0.
    
    Therefore, $$F'(\bar{x}(\cdot))(x(\cdot))=\int_0^1 x(t) \cos \bar{x}(t) dt$$ 
    is onto in $\mathbb{R}$, and so $\Ker F'(\bar{x})=T_1M(\bar{x})=\{ x(\cdot)\in C([0,1]) \mid \int_0^1 x(t) \cos \bar{x}(t) dt=0\} .$
\end{example}

The problem of description of the solution sets in more general situations (for example, nonlinear systems of inequalities) is approached qualitatively  by means of metric regularity (\cite{DoLeRo02,Io2000,ioffe})
 and via geometrical derivability (\cite{SeTa06}).

\subsubsection{A generalization of the Lyusternik theorem}
\label{sub5.1}

Consider the problem of description of the {\em solution set} $M(\bar{x})$ in the nonregular case.
As we demonstrated in Example \ref{Ex:L}, the classical Lyusternik Theorem \ref{Th:LT} might not hold if
  $F$ is singular at  $\bar{x}$, so that 
$T_1M(\bar{x})\neq  \Ker  F'(\bar{x})$.

The first generalization of the classical Lyusternik theorem for
$p$-re\-gu\-lar mappings
 was derived and proved simultaneously in \cite{BMS} and \cite{Tr83} see also \cite{IzTr94}. It can be applied
to describe the zero set of a $p$-regular mapping.


\begin{theorem}[Generalized Lyusternik Theorem, \cite{Tr83}]
\label{5th1}
Let $X$ and $Y$ be Banach spaces, and
   $U$ be a neighborhood of a point $\bar{x} \in X$. Assume that
   $F: X \to Y$ is a $p$--times continuously Fr\'echet differentiable mapping
   in $U$, it is $p$--regular at $\bar{x}$, and $$F^{(r)}(0)=0, \quad r=0,1,
   \ldots, p-1.$$

   Then
     $$
        T_1 M(\bar{x}) =  H_p(\bar{x}),
     $$
     where the set $H_p(\bar{x})$ is defined in \eqref{Hp}.
\end{theorem}

The problem of description of the tangent cone to the solution set $M(\bar{x})$ of a nonlinear equation with a singular mapping $F$ has also been studied in other papers (see, for example, \cite{BeTr17,BrTr07,LeSch98,Tr84}).
Below we present the proof of the inclusion $$T_1M(\bar{x})\subset  H_p(\bar{x})=\Ker   F^{(p)}(\bar{x}).$$ 
The proof of the inverse inclusion is based on the Multi-valued Contraction Mapping Principle (see Lemma 1, \cite{theoryofetremal}) and can be found in \cite{Tr83}.
\begin{proof}
    Without loss of generality, we assume that $\bar{x}=0$ and $F(\bar{x})=0.$
    Consider~$x\in T_1M(0)$. We will show that $x\in  \Ker   F^{(p)}(0).$ Since $x \in T_1M(0)$, by Definition \ref{Def:TC}, we have $tx+r(t)\in M(0)$, where $\|r(t)\|/t \rightarrow 0$ as $t\rightarrow 0$, and
    \begin{equation}
    \label{lyu-1}
    F(tx+r(t))=F(0).
    \end{equation}
    On the other hand,
    \begin{equation}
    \label{lyu-2}
        F(tx+r(t))=F(0)+\frac{F^{(p)}(0)}{p!}[tx+r(t)]^p+\omega(t) \quad \mbox{and} \quad \frac{\|\omega(t)\|}{t^p} \rightarrow 0 \quad \mbox{as} \quad t\rightarrow 0.
    \end{equation}
    Comparing right-hand sides of \eqref{lyu-1} and \eqref{lyu-2}, we have
    \begin{equation*}
        F(0)=F(0)+\frac{F^{(p)}(0)}{p!}[tx+r(t)]^p+\omega(t)=0.
    \end{equation*}
    Dividing by $t^p$, we obtain
    $$
\frac{F^{(p)}(0)}{p!}\left[x+\frac{r(t)}{t}\right]^p+\frac{\omega(t)}{t^p}=0.
    $$
    By the Binomial theorem and using the fact that $F^{(p)}(0)[x_1, \ldots, x_p]$ is a symmetric multilinear form, we get
    \begin{equation}
        \label{lyu-3}
        F^{(p)}(0) [x]^p + \binom{p}{1} F^{(p)}(0)\left[ \frac{r(t)}{t}, \underbrace{x,\ldots, x}_{p-1} \right] +\ldots+ F^{(p)}(0) \left[ \frac{r(t)}{t}\right]^p + \frac{\omega(t)}{t^p}p!=0.
    \end{equation}
Since $F^{(p)}(0)$ is continuous, $|r(t)\|/t \rightarrow 0$ and $\dfrac{\|\omega(t)\|}{t^p} \rightarrow 0$ as  $t\rightarrow 0$, we get 
 $$F^{(p)}(0)[x]^p=0.$$ Therefore, $x\in \Ker  F^{(p)}(0)$ and  $T_1M(0)\subset \Ker  F^{(p)}(0).$
\end{proof}

\begin{example}
\label{5ex4}
To illustrate the statement of Theorem \ref{5th1}, 
let mapping $F$ be defined by
\begin{equation} \label{F=exx}
 F(x)=
\left(
\begin{array}{c}
  x_1^2-x_2^2+x_3^2 \\
   x_1^2-x_2^2+x_3^2+x_2x_3 \\
\end{array}
\right).
\end{equation}
Consider $\bar{x}=(0,0,0)^T.$ It is easy to verify that $F'(\bar{x})=0$,
$$F''(\bar{x})= \left(
\begin{array} {c} \left(
                  \begin{array}{ccc}
 2 & 0& 0 \\
 0 & -2 & 0\\
 0 & 0 & 2
                  \end{array} \right)\\
                  \\
\left(
\begin{array} {ccc}
 2 & 0& 0 \\
 0 & -2 & 1\\
 0 & 1 & 2
\end{array}
\right)
                  \end{array} \right).
                  $$
Also,
 $$
  \operatorname{Ker}^2 F''(x^*)=\operatorname{span}\left( \left[ \begin{array}{c} 1\\ -1 \\0 \end{array} \right] \right) \bigcup \operatorname{span}\left( \left[ \begin{array}{c} 1\\ 1 \\0 \end{array} \right] \right).$$

 Let $h=(1,1,0)^T$ (or $h=(1,-1,0)^T$), then $\operatorname{Im}  F''(\bar{x})h=\mathbb{R}^2.$
Hence, the mapping $F(x)$ is $2$-regular at $\bar{x}=0$.
 Then  the statement of Theorem \ref{5th1} in this example reduces to  $$T_1M(\bar{x})=H_2(\bar{x})=\Ker^2   F^{\prime \prime}(\bar{x})$$ 
or
$$T_1M(\bar{x}) = \Ker^2 F''(\bar{x})=\operatorname{span}\left( \left[ \begin{array}{c} 1\\ -1 \\0 \end{array} \right] \right) \bigcup \operatorname{span}\left( \left[ \begin{array}{c} 1\\ 1 \\0 \end{array} \right] \right).$$
\end{example}

Next theorem presents another version of \ref{5th1}, which was formulated in \cite{Tr87}. (See also \cite{BMS} and \cite{IzTr94}
for additional results along these lines.) 
To state the result, we  denote by $\hbox{dist}(x,M)$, the \emph{distance function} from
a point $x \in X$ to a set $M$:
\vskip-6pt
$$\hbox{dist}(x,M)=\inf_{y\in M}\|x-y\|, \quad x\in X.$$

\begin{theorem}[\cite{PrTre16}]
\label{5th3}
Let $X$ and $Y$ be Banach spaces, and $U$ be a neighborhood of a point $\bar{x}\in X$. Assume that $F:X\rightarrow Y$ is a $p$-times continuously Fr\'{e}chet differentiable mapping in $U$ and satisfies the condition of strong $p$-regularity at $\bar{x}$. Then there
exist a neighborhood $U'\subseteq U$ of $\bar{x}$, a mapping $\xi\rightarrow x(\xi):U'\rightarrow X$, and constants $\delta_1>0$ and $\delta_2>0$ such that
$ F(\xi+x(\xi))=F(\bar{x}),$
\begin{equation}
\label{5eq1}
{\rm dist} (\xi, M(\bar{x})) \leq   \|x(\xi)\|_X \leq \delta_1 \sum_{i=1}^p\dfrac{\|f_i(\xi)-f_i(\bar{x})\|_{Y_i}}{\|\xi-\bar{x}\|^{i-1}}
\end{equation}
and
$$ {\rm dist} (\xi, M(\bar{x})) \leq  \|x(\xi)\|_X \leq \delta_2 \sum_{i=1}^p\|f_i(\xi)-f_i(\bar{x}) \|^{1/i}_{Y_i}
   $$
 for all $\xi\in U'$.
\end{theorem}
For the proof, see \cite{IzTr94} and \cite{PrTre16}.



\subsubsection{Representation theorem}
\label{sub5.2}
The Representation theorem is used in nonlinear analysis 
and is relevant to the study of local behavior and representation of a mapping $F$ around some special point $\bar{x}$.
The theorem also guarantees the existence of some auxiliary mappings that have desirable properties and relate to the given mapping $F$ and its local representation in some neighborhood of $\bar{x}$.

The Representation theorem can be used, for example, in the study of optimization problems, particularly in constrained optimization. It helps in establishing the existence of critical points and characterizing their properties, which is crucial in finding optimal solutions.
The theorem is also relevant to variational methods, 
partial differential equations, and other areas.
The theorem is useful in various numerical methods and computational techniques for approximating solutions of equations.
Its versatility and usefulness stem from its ability to provide insights into the local behavior and representations of mappings around critical points, which has wide-ranging applications in mathematical analysis and optimization.

\begin{theorem}[\cite{TrMa03}]
\label{5th5}
Let $X$ and $Y$ be Banach spaces, and $V$ be a neighborhood of $\bar{x}$ in $X$. Suppose  $F:V\rightarrow Y$ is of class $\mathcal{C}^{p+1}$ and $F^{(i)}(\bar{x})=0,$ $i=1,\ldots,p-1$. Suppose further that for a constant $C>0,$
$$\sup_{\|h\|=1}\left\|\{F^{(p)}(\bar{x})[h]^{p-1}\}^{-1}\right\|\leq C.$$
Then there exist a neighborhood $U$ of $0$ in $X$, a neighborhood $V$ of $\bar{x}$ in $X$, and mappings $\varphi:U\rightarrow X$ and $\psi:V\rightarrow X$ such that $\varphi$ and $\psi$ are Fr\'{e}chet-differentiable at $0$ and $\bar{x}$, respectively, and
\begin{enumerate}
  \item $\varphi(0)=\bar{x},$ $\psi(\bar{x})=0$;
  \item $F(\varphi(x))=F(\bar{x})+\dfrac{1}{p!}F^{(p)}(\bar{x})[x]^p$ for all $x\in U$;
  \item $F(x)=F(\bar{x})+\dfrac{1}{p!}F^{(p)}(\bar{x})[\psi(x)]^p$ for all $x\in V$;
  \item $\varphi'(0)=\psi'(\bar{x})=I_X$.
\end{enumerate}
\end{theorem}

Note that if we introduce mapping $\Phi(\xi)$,
\begin{align*}
     \Phi(\xi) &= F(x)-F(\bar{x})+P_1F'(x)[\xi-(x-\bar{x})]
  \\
  &+\dfrac{1}{2}P_2F''(x)[\xi-(x-\bar{x})]^2+ \ldots +\dfrac{1}{p!}P_{p}F^{(p)}(x)[\xi-(x-\bar{x})]^p
\end{align*}
 and apply Generalized Lyusternik Theorem \ref{5th3},
we get that there exist mappings $\varphi(x)=\bar{x}-x+\xi(x)$, $\varphi: U\rightarrow X$, and $\psi(x)=\bar{x}-x+\eta(x)$,  $\psi: V\rightarrow X$, such that Representation Theorem  \ref{5th5} holds.

All assumptions of Theorem \ref{5th5} hold, for example, for the mapping 
 $$F(x) = x_1^p - x_2^p + x_1^{p+1} + x_2^{p+1},$$ where $p \geq 2$, $p \in {\mathbb N}$. 
See additional work on  the representation theorem in paper \cite{Iz2000}.

\subsubsection{Morse Lemma}
\label{sub5.3}

The Morse Lemma is another fundamental result 
in analysis, which relates the behavior of a smooth function near some non-degenerate critical point $\bar{x}$ to the local behavior of its level sets. The lemma has several important applications in various areas of mathematics.

The Morse Lemma is used in differential geometry to analyze the behavior of geodesics and study the geometry of manifolds. By considering a function that measures the length or energy of curves on a manifold, the Morse Lemma allows us to understand the critical points of this function and their geometric implications. It provides insights into the existence, stability, and bifurcations of geodesics on a manifold.

The Morse Lemma has applications in optimization and control theory, where it is used to analyze the behavior of objective functions and control systems near critical points. It helps in characterizing the local behavior of optimal solutions and understanding stability properties. The lemma can be employed in finding critical points, performing sensitivity analysis, and studying bifurcations in optimization problems and dynamical systems.

The Morse Lemma is also utilized in singularity theory, which studies the properties and classifications of singular points or critical points of differentiable mappings. It provides a framework for understanding the local behavior of singularities and the possible changes in their structure under small perturbations. The lemma is instrumental in the classification and analysis of singular points and their stability properties.

The most interesting formulation of the Morse Lemma for the finite dimensional case is given in the following lemma.

\begin{lemma}[Morse Lemma]
\label{5lem6}
  Let $\bar{x}\in \mathbb{R}^n$ and $f:\mathbb{R}^n \rightarrow \mathbb{R},$ $f\in \mathcal{C}^3$, be such that $f'(\bar{x})=0$, but $f''(\bar{x})$ is not degenerate. Then in a neighborhood $V$ of $\bar{x}$ there exist a curvelinear coordinate system $(y_1,\ldots,y_n)$ and an integer number $k\in \{0,\ldots,n\}$ such that
  $$f(x)=f(\bar{x})+\sum_{i=1}^ky_i^2-\sum_{i=k+1}^n y_i^2$$
  for all $x\in V$.
\end{lemma}

\begin{proof}
  Without loss of generality, we can assume that Hessian matrix $f''(\bar{x})$ is diagonal:
  $$
  f'' (\bar{x}) = \left( \begin{array} {ccccc}
  1 & 0 & \cdots & 0 & 0\\
  0 & 1 & \cdots & 0 & 0\\
    \vdots & \vdots& \ddots &     \vdots &     \vdots\\
  0 & 0 & \cdots & -1 & 0\\
    0 & 0 & \cdots & 0 & -1\\
  \end{array}
  \right),
  $$
  where for some number  $k'$ between $0$ and $n$, the first $k'$ columns have 1 on the main diagonal, and the other columns have $-1$.
     Otherwise, changing the basis, we can transform the Hessian to be a diagonal matrix.

Then in this case,
    $$f(x)=f(\bar{x})+\sum_{i=1}^{k'}(x_i-x_i^*)^2-\sum_{i=k'+1}^n (x_i-x_i^*)^2+o(\|x-\bar{x}\|^3).$$
Note, that if the assumptions of Morse Lemma  \ref{sub5.3} hold, then the assumptions of the representation \ref{5th5} are satisfied with $p=2$ and $F=f$.
Hence, there exists a mapping $\psi(x):U\rightarrow \mathbb{R}$ such that $$f(\psi(x))=f(\bar{x})+\dfrac{1}{2}f''(\bar{x})[x]^2,$$
 where $\|\psi(x)-(x-\bar{x})\|=o(\|x-\bar{x}\|)$ and $\psi'(\bar{x})=I_X$. It follows that $k=k'$. Note, that if $k=k'$ then $\psi(x)=x-\bar{x}+o(x-\bar{x})$ and if $k\neq k'$ then we get a contradiction.
 Now, we can apply the statement of the representation  \ref{5th5} to the mapping $y=\psi(x)$ to get the statement of Morse Lemma \ref{5lem6}.
\end{proof}

See additional work on Morse Lemma  in \cite{Iz98}.

\subsection{Implicit Function Theorem}
\label{sec:IFT}

In this section, we consider the equation $F (x, y) =0$, where $F : X \times Y \to Z$ and $X$, $Y$ and $Z$ are Banach spaces.
 Let $(\bar{x},  \bar{y})$ be a given point in $X \times Y$ that satisfies $F(\bar{x},\bar{y}) = 0$. 
We are interested in  the problem of existence of a locally defined mapping $\varphi(x)$, which is a solution of the equation $F (x, y) =0$ near the given point $(\bar{x},  \bar{y})$, so that the following holds:
\begin{equation} \label{E12}
F(x, \varphi(x))=0 \quad \mbox{and} \quad \bar{y} = \varphi(\bar{x}).
\end{equation}

\subsubsection{Implicit Function Theorem in the regular case}
\label{sub3.7}

In the case when $F : X \times Y \to Z$ is a continuously differentiable mapping that is regular at the point $(\bar{x}, \bar{y})$, i.e., when $F_y^\prime (\bar{x}, \bar{y})$, the derivative of $F$ with respect to $y$, is onto, the classical implicit function theorem guarantees
 existence of the smooth mapping $\varphi(x)$ defined in a neighborhood $U(\bar{x})$ of $\bar{x}$ such that \eqref{E12} holds and $\|\varphi(x) - \bar{y}\|\leq C \|F(x,\bar{y})\|$ for all $x$ in $U(\bar{x})$, where $C>0.$
There are numerous books and papers devoted to the implicit function theorem, amongst which are \cite{Hamilton, Krantz}. 
There are various statements of the standard implicit function theorem and Theorem \ref{IFTL} is one of them.

\begin{theorem}[Implicit Function Theorem]
\label{IFTL}
Let $X$ and $Y$ be  normed linear spaces.
Assume that
$F : X\times Y \to {\mathbb R}^m$ is continuously Frechet differentiable at $(\bar{x}, \bar{y}) \in X\times Y$, $F(\bar{x}, \bar{y})=0$, and
that
$F^\prime_y(\bar{x}, \bar{y})$ is onto.
Then there exist constants $C, C_1 >0$, a sufficiently small $\delta$,
and a function
$y: B (\bar{x}, \delta) \to  Y$
such that the following holds for $x \in B (\bar{x}, \delta)$:
\begin{equation} \label{IFTE}
	\varphi(\bar{x})=\bar{y}, \quad
	F(x, \varphi(x))=0, \quad
	\|\varphi(x)-\bar{y}\|\leq C_1 \| F(x, \bar{y}) \|\leq C \| x - \bar{x}\|.
\end{equation}
\end{theorem}

The situation changes when the mapping F is degenerate (nonregular) at $(\bar{x}, \bar{y})$; that is, when  $F_y^\prime (\bar{x}, \bar{y})$ is not onto, and, hence, the classical implicit function theorem can not be applied to guarantee the (local) existence of a solution $y(x)$. The importance of consideration of this situation follows from the need of solving various nonlinear problems, many of which, as was shown in \cite{TrMa03}, are, by their nature, singular (degenerate).

\subsubsection{Implicit Function Theorem in the degenerate case}
\label{sub5.8}

In this section, we focus on the case when mapping $F : X \times Y \to Z$  is not regular, that is when $F_y^\prime (\bar{x}, \bar{y})$ is not onto. 

As an example, consider mapping $F:\mathbb{R}\times \mathbb{R}\rightarrow \mathbb{R}$, $F(x,y)=x-y^p$, where $p=2k+1$ with some $k\in \mathbb{N}$. If  $(\bar{x},\bar{y})=(0,0)$ then $F(\bar{x}, \bar{y})=0$ and $F_y^\prime (\bar{x}, \bar{y})=0$,  so that mapping $F$ is not surjective.
The classical implicit function theorem is not applicable in this case. However, there exists  mapping  $\varphi(x)=x^{1/p}$ such that
$F(x,x^{1/p})=x-(x^{1/p})^p=0$ and, moreover, $\|\varphi(x) -  \bar{y}\|= \|F(x,\bar{y})\|^{1/p}$.  Hence, the following inequality holds with $C>0$:  $$\|\varphi(x)- \bar{y}\|\leq C \|F(x,\bar{y})\|^{1/p}.$$
Thus, while the conditions of a standard implicit function are not satisfied in the example, the statement similar to equations \eqref{IFTE} holds. The example serves as a motivation and illustration for the $p$-order implicit function theorems. 
To our knowledge, the first generalization of the implicit function theorem for nonregular mappings was formulated in \cite{Tr87}. Generalizations of the implicit function theorem for 2-regular mappings were obtained in \cite{BrTr07, Iz98, IzTr94, IzTr99}. We will present a few versions of the implicit function theorem for $p$-regular mappings in this section.

To simplify the presentation, we start with Theorem \ref{5th17} that is stated in Euclidean spaces. A slight modification of Theorem \ref{5th17} was derived in \cite{IzTr94} and \cite{Tr-diss}.
To formulate the theorem, we define an operator $ \Psi_{p\, y}(\bar{x}, \bar{y}, h) $ as follows:
$$
    \Psi_{p\, y}(\bar{x}, \bar{y}, h)  = f^\prime_{1\, y} (\bar{x}, \bar{y}) + f^{ \prime \prime}_{2\, yy} (\bar{x}, \bar{y})h + \cdots + f_{p\, y\cdots y}^{(p)} (\bar{x} \bar{y}) [h]^{p-1}.
$$

\begin{remark}
Note that the definition of $ \Psi_{p\, y}(\bar{x}, \bar{y}, h) $ is similar to the definition of the operator  $\Psi_p(h)$ in \eqref{4eq3}.
\end{remark}

\begin{theorem}[Implicit Function Theorem \cite{TrMa03}]
\label{5th17}
Suppose that  $X$, $Y$ and $Z$ are Euclidean
spaces, $W$ is a neighborhood of a point $(\bar{x}, \bar{y})$ in $X\times Y$, and assume that $F : W \rightarrow Z$
is of class $\mathcal{C}^2$. Suppose $F(\bar{x}, \bar{y}) = 0$ and the following conditions hold:
\begin{itemize}
\item[{\rm 1.}]
{\bf  the singularity condition:}
  $${f_i^{(r)}}_{\underbrace{\xi\ldots \xi}_q \underbrace{y\ldots y}_{r-q}}(\bar{x}, \bar{y})=0, \quad r=1,\ldots, i-1, \; q=0,\ldots, r-1, \; i=1,\ldots, p;$$
\item[{\rm 2.}]
{\bf   the $p$-regularity condition} at the point $(\bar{x}, \bar{y})$: there is a neighborhood $U(\bar{x})$ of $\bar{x}$ in $X$ such that
  $$
         \Psi_{p\, y}(\bar{x}, \bar{y}, h)  Y = Z
      $$
for all
      $$		
        h  \in \{ \Psi_{p\, y}(\bar{x}, \bar{y}) \}^{-1}
         (- F(\xi, \bar{y}))
$$
  and all $\xi\in U(\bar{x})$ such that $F(\xi,\bar{y})\neq 0;$
\item[{\rm 3.}]
{\bf  the Banach condition:} for any $z\in Z$, $\|z\|=1$, there exists $y\in Y$ such that
  $$  \Psi_{p\, y}(\bar{x}, \bar{y}, y) y = z,
         \quad \|y\| \leq c,
      $$
  where $c> 0$ is independent of $z$ constant;
\item[{\rm 4.}]
{\bf  the elliptic condition} with respect to the independent variable $\xi$:
  $$\|f_i(\xi,\bar{y})\|_{Z_i}\geq m\|\xi-\bar{x}\|_X^i$$
    for all $\xi\in U$ and for all $i=1,\ldots,p$, where $m>0$ in some number and $U$ is a neighborhood of the point $\bar{x}$ in $X$.
\end{itemize}

Then for any $\varepsilon>0$ there exist $\delta>0$, $K>0$ and a map $\varphi: U(\bar{x}, \delta)\rightarrow U(\bar{y}, \varepsilon)$ such that
\begin{itemize}
\item[{\rm (a)}] $\varphi(\bar{x})=\bar{y}$;
\item[{\rm (b)}] $F(\xi,\varphi(\xi))=0$ for all $\xi\in U(\bar{x}, \delta)$;
\item[{\rm (c)}] $\|\varphi(\xi)-\bar{y}\|_Y\leq K \mathop{\sum}\limits_{i=1}^{p}  \left\|f_i(\xi,\bar{y})\right\|_{Z_i}^{1/i}$
  for all $\xi \in U(\bar{x}, \delta)$.
\end{itemize}
\end{theorem}

Another version of the implicit function theorem for nonregular mapprings, which is presented as Theorem \ref{GIFT}, was proved in  \cite{BTM08}.
To  state the theorem, we define the following mappings (see \cite{Tr87}), which are similar to definitions of mappings $f_i(x)$, given in \eqref{4eq2} for 
operator  $F : X \rightarrow Y$: 
\begin{equation} \label{fi}
   f_i(x, y): X \times Y \to Z_i, \quad
   f_i(x, y) = P_{Z_i} F(x, y), \quad i=1,\,\dots,p,
\end{equation}
where
$P_{Z_i} : Z \to Z_i$ is the projection operator onto
$Z_i$ along
$(Z_1 \oplus \ldots \oplus Z_{i-1} \oplus Z_{i+1} \oplus \ldots \oplus Z_p)$
with respect to $Z$, $i=1,\,\dots,p$.

Let us recall that the mapping $F$ is called {\em uniformly p-regular} over the set $M$ (see Definition 2.3 in \cite{+2007}) if 
$$
\sup\limits_{h\in M} \| \Psi_p(\bar{h}^{-1}  \| <\infty,\ \  \bar{h}=\frac{h}{\|h\|}, h\neq 0,
$$
where $\| \Psi_p(\bar{h}^{-1}  \| =\sup_{\|z\|=1} \inf \{ \Psi(\bar{h})[y]=z \}$.

\begin{theorem}[The $p$th-order Implicit Function Theorem]
\label{GIFT}
Let $X$, $Y$ and $Z$ be Banach spaces, $N(\bar{x})$ and
$N(\bar{y})$ be sufficiently
small neighborhoods of $\bar{x} \in X$ and $\bar{y} \in Y$ respectively,
$F\in C^{p+1}(X\times Y)$, and $F(\bar{x}, \bar{y})=0$. Let
the mappings $f_i(x, y)$, $i=1, \ldots, p$, introduced in equation
$\eqref{fi}$,  satisfy the following conditions:

{\rm 1)} {\bf singularity condition}:
$$
 	\begin{array}{l}
	\!\!\!\!\!\!
	f^{(r)}_{i\;\mbox{\scriptsize $\underbrace{\strut x\dots x}_q\;
	\underbrace{\strut y\dots y}_{r-q}$}}(\bar{x}, \bar{y})=0,
	\quad r=1,\dots,i-1,\ \,q=0,\dots,r-1,\ \,i=1,\dots,p, \\
	\!\!\!\!\!\!
	f^{(i)}_{i\;\mbox{\scriptsize $\underbrace{\strut x\dots x}_q\;
	\underbrace{\strut y\dots y}_{i-q}$}}(\bar{x}, \bar{y})=0,
	\quad q=1,\dots,i-1,\ \,i=1,\dots,p;
	\end{array}
$$

{\rm 2)} {\bf $p$-factor-approximation}:
for all $y_1, y_2 \in (N(\bar{y})-\bar{y})$,
\begin{align*}
	&	\left\| f_i(x,\bar{y}+y_1)-f_i(x,\bar{y}+y_2)-\dfrac{1}{i!}
	f^{(i)}_{i \, y\ldots y} (\bar{x}, \bar{y})[y_1]^i+
	\dfrac{1}{i!}
	f^{(i)}_{i \, y\ldots y} (\bar{x}, \bar{y}) [y_2]^i \right\| \\
	& \quad	\leq \varepsilon\left(\|y_1\|^{i-1}+\|y_2\|^{i-1}\right)\|y_1-y_2\|,
	\qquad i=1,\dots,p,
\end{align*}
where $\varepsilon>0$ is sufficiently small;

{\rm 3)} {\bf Banach condition}:
there exists a nonempty set
$\Gamma(\bar{x}) \subset  N(\bar{x})$ in $X$ such that
for any sufficiently small $\gamma$, we have $\Gamma(\bar{x}) \cap B_{\gamma}(\bar{x}) \neq \{\bar{x}\}$, where $B_{\gamma}(\bar{x})$ is a ball of radius $\gamma$ with the center $\bar{x}$.
Moreover, for $x \in \Gamma(\bar{x})$, there exists $h(x)$ such that
\begin{equation} \label{BanachC}
          \Phi_p[h(x)]^p=-F(x,\bar{y}), \quad
	\|h(x)\|\leq c_1 \sum\limits_{r=1}^p\|f_r(x,\bar{y})\|_{Z_{r}}^{1/r},
\end{equation}
where $0< c_1 < \infty$ is a constant;

{\rm 4)} {\bf uniform $p$-regularity condition}
of the mapping $F(x, y)$ over the set
$$\Phi_p^{-1} (-F(x, \bar{y})).$$
Moreover, assume that for any sufficiently small $\gamma$ such that
$B_{\gamma}(\bar{x}) \subset  N(\bar{x})$ the intersection of $\Gamma(\bar{x})$ and
$B_{\gamma}(\bar{x})$ is not empty.
\smallskip

Then there exists a constant $k>0$, a sufficiently small $\delta$,
and a mapping
$\varphi: \Gamma(\bar{x}) \cap B_\delta(\bar{x}) \to N(\bar{y})$
such that the following hold for $x \in  \Gamma(\bar{x}) \cap B_\delta(\bar{x})$:
$$
	\varphi(\bar{x})=\bar{y};
$$
$$
	F(x,\varphi(x))=0,
$$
$$
	\|\varphi(x)-\bar{y}\|_Y\leq
	k\sum\limits_{r=1}^p\|f_r(x,\bar{y})\|_{Z_{r}}^{1/r}.
$$
\end{theorem}

There are generalizations of implicit function theorems for nonregular mappings derived by other authors. Some examples include a generalization of the implicit function theorem and its application to a parametric linear time-optimal control problem presented in \cite{Ko2006}, generalized implicit function theorems applied to ordinary differential equations in \cite{Ar2020}, and implicit function theorems for 2-regular mappings in \cite{Ar2010, Ar2021}. 


\subsection{Newton's method}
\label{sec:Newton}

\subsubsection{Classical Newton's method for nonlinear equations and unconstrained optimization problems}
\label{sub3.3}

Consider the problem of solving nonlinear equation \eqref{F(x)=0},
where $F:X\rightarrow Y$ is sufficiently smooth, so that 
$F\in \mathcal{C}^{p+1}(X)$ for some $ p\in \mathbb{N}.$ Let $\bar{x}$ be a~solution
to  Equation \eqref{F(x)=0}, that is $F(\bar{x})=0$. Assume that mapping $F$ is singular at the point $\bar{x}$.

In the finite dimensional case, when $F(x)=(F_1(x),\ldots, F_n(x))^T$, $X=\mathbb{R}^n$, and $Y=\mathbb{R}^n$, the singularity of $F$ at $\bar{x}$ means that the Jacobian $F'(\bar{x})$  of $F$ 
 is singular at $\bar{x}$, as in the following  example.

 \begin{example}[\cite{SzPrTr12}]
 \label{3ex11}
Consider function $F:\mathbb{R}^2 \rightarrow\mathbb{R}^2$ from Example \ref{3ex1} defined by 
$$
    F(x)=\left(
           \begin{array}{c}
             x_1+x_2 \\
             x_1x_2 \\
           \end{array}
         \right),
$$
where $\bar{x}=(0,0)^T$ is a
solution to Equation \eqref{F(x)=0} and $
    F'(\bar{x})=\left(
           \begin{array}{cc}
             1 & 1 \\
             0& 0 \\
           \end{array}
         \right)
$ is singular (degenerate) at the point $\bar{x}.$

Consider a sufficiently small $\varepsilon>0$ and some 
initial point $ x^0 \in B(0, \varepsilon)$. 
The classical Newton method is defined by
\begin{equation}
\label{3eq1}
    x^{k+1}=x^{k}-\left( F'(x^k) \right)^{-1} F(x^k), \quad k=0,1,\ldots .
\end{equation}
If $x^k = (x_1, x_2)$, we get in this example
$$    F'(x^k)=\left(
           \begin{array}{cc}
             1 & 1 \\
             x_2 & x_1 \\
           \end{array}
         \right), \quad
           \left( F'(x^k) \right)^{-1}=\dfrac{1}{x_1- x_2}\left(
           \begin{array}{cc}
             x_1 & -1 \\
            - x_2 & 1 \\
           \end{array}
         \right) .
$$
Then
$$
         \left(F'(x^k) \right)^{-1}F(x^k) =
         \dfrac{1}{x_1- x_2}
         \left(
           \begin{array}{cc}
 x_1^2 \\
 -x_2^2
           \end{array}
         \right), 
         $$
         $$
                \quad
         x^{k+1} =
         x^{k} - \left( F'(x^k) \right)^{-1}F(x^k) =
         \dfrac{1}{x_1- x_2}
         \left(
           \begin{array}{cc}
- x_1  x_2  \\
 x_1 x_2
           \end{array}
         \right) .
$$

If $x_{1}=x_{2}$, then  $\left( F'(x^k) \right)^{-1}$ does not exist and, hence, method \eqref{3eq1} is not applicable.
Even in the case when $\left( F'(x^k) \right)^{-1}$ exists,  method \eqref{3eq1} might be diverging.
As an example, consider point 
$x^k=(t+t^3,t)^T$, where $t$ is sufficiently small. Then $$x^{k+1}=
 \dfrac{1}{t^3}
         \left(
           \begin{array}{cc}
- t^2 - t^4 \\
t^2+t^4
           \end{array}
         \right) = \left(-\dfrac{1}{t}-t,
\dfrac{1}{t}+t \right)^T$$ and, for a sufficiently small values of $t$,  $\|x^{k+1}-\bar{x}\| = \|x^{k+1}-0\|\approx\dfrac{1}{t}\rightarrow\infty$ when $t\rightarrow 0^+.$
For instance, if $t=10^{-5}$, then $\|x^{k+1}-0\|\approx10^5$ and the method \eqref{3eq1} is diverging.
\end{example}

\begin{example}\cite{Reddien78}
\label{3ex2}
Let mapping $F:\mathbb{R}^2 \rightarrow\mathbb{R}^2$ be defined by 
$$
F(x)=\left(\begin{array}{l}
x_{1}+x_{1}x_{2}+x_{2}^{2}\\
x_{1}^{2}-2x_{1}+x_{2}^{2}
\end{array}\right).
$$
In this example,  singular  root of $F(x)=0$  is $\bar{x}=(0,0)^{T}$.
It is easy to verify that
$\Ker   F(\bar{x})=\operatorname{span}\{(0,1)\}$
and 
$\operatorname{Im}  F(\bar{x})=\operatorname{span}\{(1,-2)\}.$
Also, the Jacobian $F'(x)$ is  singular on the hyperbola given by
$$
2x_1-2x_{1}^{2}+6x_{2}-4x_{1}x_{2}+2x_{2}^{2}=0.
$$
\end{example}
For the overview of the existing approaches to Newton-like methods for singular operators, see e.g.
\cite{BKL2010}.

Now we consider Newton's method for finding critical points of an unconstrained optimization problem:
\begin{equation}
\label{NMopt}
\min_{x\in\mathbb{R}^{2}}f(x),
\end{equation}
where $ f:\mathbb{R}^2 \rightarrow \mathbb{R}$.
Classical Newton's method applied to problem \eqref{NMopt} has the form
\begin{equation}
\label{3eq2}
    x^{k+1}=x^{k}- ( f''(x^k) )^{-1}f'(x^k).
\end{equation}
As an example, consider minimization of function $f$ given by 
$ f(x)=x_1^2+x_1^2x_2+x_2^4$ (see \cite{SzPrTr12}).
One of the critical points of the function $f$ is $\bar{x}=(0,0)^T$.
Let
$x^0=(x^{0}_{1},x^{0}_{2})^T$ where $x^{0}_{1}=x^{0}_{2}\sqrt{6(1+x^{0}_{2})}$. Then
$$
f''(x^0)=\left(%
\begin{array}{cc}
  2+2 x^{0}_{2} & 2 x^{0}_{2} \sqrt{6(1+x^{0}_{2})} \\
  2x^{0}_{2} \sqrt{6(1+x^{0}_{2} )} & 12 (x^{0}_{2})^2 \\
\end{array}%
\right)
$$ and
$\det f''(x^0)=0.$  Hence, $( f''(x^0) )^{-1}$  does not exist, so Newton's method \eqref{3eq2} is not applicable.

\subsubsection{The $p$-factor Newton's method}
\label{sub5.5}

In this section, we describe a method for solving nonlinear equations given in the form \eqref{F(x)=0}, where
$F:\mathbb{R}^n\rightarrow \mathbb{R}^n$ and the matrix $F'(\bar{x})$ is singular at the solution point
$\bar{x}$ (see \cite{SzPrTr12}). The proposed method is based on the construction of the $p$-factor operator.

 There are various publications describing the p-factor-method for solving degenerate nonlinear systems and nonregular optimization problems. Some examples are \cite{BeTr88, BrEvTr06, CGA2023, SzPrTr12, SzTr08}.

Let $h \in \mathbb{R}^n$, 
$$Y_1=\hbox{Im} F'(\bar{x}),\quad
\bar{P}_1=P_{Y_1^{\perp}}, \quad 
Y_2=\hbox{Im} \left(F'(\bar{x})+\bar{P}_1F''(\bar{x})h\right), \quad \mbox{and} \quad
\bar{P}_2=P_{Y_2^{\perp}}.$$
For each $k=2,\ldots,p-1,$ define 
$\bar{P}_{k+1}=P_{Y_{k+1}^{\perp}}$ and  
$$Y_{k+1}=\operatorname{Im}  \left(F'(\bar{x})+ \sum\limits_{i=1}^k
\bar{P}_iF''(\bar{x})h+\sum\limits_{%
\begin{smallmatrix}
  i_2>i_1 \\
  i_1,i_2\in \{1,\ldots,k \}
\end{smallmatrix}} \bar{P}_{i_2}\bar{P}_{i_1}F^{(3)}(\bar{x})[h]^2+\cdots \right.$$
$$\left.+\sum\limits_{%
\begin{smallmatrix}
  i_k>\ldots >i_1 \\
  i_1,\ldots ,i_k \in \{1,\ldots,k \}
\end{smallmatrix}}\bar{P}_{i_k}\ldots \bar{P}_{i_1}F^{(k)}(\bar{x})[h]^{(k-1)}
\right).$$ 

Let $h$ be a fixed vector such that   $\|h\|=1$ and
mapping $F$ is $p$-regular at the solution $\bar{x}$ along vector $h$.
Let matrices  $P_i, i= 1,\ldots,p-1$, be defined as follows:
$$P_1= \sum\limits_{i=1}^{p-1}\bar{P}_i, \quad P_2=\sum\limits_{%
\begin{smallmatrix}
  i_2>i_1 \\
  i_1,i_2\in \{1,\ldots,p-1 \}
\end{smallmatrix}}\bar{P}_{i_2}\bar{P}_{i_1}, $$$$
P_{k+1}=\sum\limits_{%
\begin{smallmatrix}
  i_k>\ldots>i_1 \\
  i_1,\ldots,i_k \in \{ 1,\ldots,p-1 \}
\end{smallmatrix}}\bar{P}_{i_k}\ldots\bar{P}_{i_1},\quad  k= 2,\ldots,p-1.$$

Consider the equation 
    $$F(\bar{x})+P_1F'(\bar{x})h+\ldots+P_{p-1}F^{(p-1)}(\bar{x})[h]^{p-1}=0.$$
Note that under these assumptions,
the $p$-factor matrix given by 
\begin{equation}
\label{5eq6}
\left(F'(\bar{x})+P_1F''(\bar{x})h+\ldots+P_{p-1}F^{(p)}(\bar{x})[h]^{p-1}\right)
\end{equation}
is not singular. Hence, 
$\bar{P}_p=0 $ and $Y_p=\mathbb{R}^n.$

 Then the $p$-factor Newton method can be defined as
\begin{equation}
\label{5eq5}
    x^{k+1}=x^k-\left(F'(x^k)+P_1F''(x^k)h+\ldots+P_{p-1}F^{(p)}(x^k)[h]^{p-1} \right)^{-1}
\end{equation}
$$
  \qquad    \qquad    \qquad   {\times} \left(F(x^k)+P_1F'(x^k)h+\ldots+P_{p-1}F^{(p-1)}(x^k)[h]^{p-1} \right).
$$

In the case of $p=2$, the $2$-factor-Newton method reduces to the following:
\begin{equation}
\label{5eq7}
    x^{k+1}=x^k- \left(F'(x^k)+P_1F''(x^k)h\right)^{-1} (F(x^k)+P_1F'(x^k)h)
\end{equation}
where $P_1$ is the orthoprojection onto $\hbox{Im}
(F'(\bar{x}))^{\perp}$ and vector $h$ $(\|h\|=1)$ is chosen in  such a way
that the $2$-factor matrix
\begin{equation}
\label{5eq8}
\left(    F'(\bar{x})+P_1F''(\bar{x})h \right)
\end{equation}
is not singular. In fact, it means that $F$ is $2$-regular along $h$.
Then the equation
$$
    F(\bar{x})+P_1F'(\bar{x})h=0
$$
is satisfied at $\bar{x}$. 
Note that by \ref{5eq8},
the point  $\bar{x}$ is a locally unique solution of the nonlinear equation  \eqref{F(x)=0}.

The $2$-factor Newton method can be applied to solve the equation in \ref{3ex11}.

\begin{theorem}[\cite{SzPrTr12}]
\label{5th9}
Let $F\in \mathcal{C}^p(\mathbb{R}^n)$ and there exists $h,\|h\|=1$, such
that the $p$-factor matrix defined in \ref{5eq6} is not singular. Then for
any $x^0 \in U_\varepsilon(\bar{x})$ (with $\varepsilon>0$
sufficiently small) and for the sequence $\{x^k\}$ generated by applying method \ref{5eq5}, the following inequality holds for some constant $c>0$:
\begin{equation}
\label{2.8}
    \|x^{k+1}-\bar{x}\|\leq c\|x^{k}-\bar{x}\|^2, k=0,1,\ldots .
\end{equation}
\end{theorem}

\begin{example}[\cite{SzPrTr12}]
\label{5ex10}
Let
$
    F(x)=\left(
           \begin{array}{c}
             x_1+x_2 \\
             x_1x_2 \\
           \end{array}
         \right),
$
$\bar{x}=(0,0)^T$. It was shown in  \ref{3ex1} and  \ref{3ex11} that $F$ is singular at $\bar{x}=(0,0)^T.$
To apply the $2$-factor Newton method given in \ref{5eq7} to this problem, we first define the $2$-factor-operator in the form:
$$F'(x^k)+P_1F''(x^k)h=\left(
                         \begin{array}{cc}
                           1 & 1 \\
                           x^k_2-1 & x^k_1+1\\
                         \end{array}
                       \right),$$
                       where
  $P_1=\left(
           \begin{array}{cc}
             0 & 0 \\
             0 & 1 \\
           \end{array}
         \right)$ and $h=(1,-1)^T.$
Then the $2$-factor Newton method \ref{5eq7} in this example is defined as follows:

$$
    x^{k+1}=x^k-\left(
                         \begin{array}{cc}
                           1 & 1 \\
                           x^k_2-1 & x^k_1+1\\
                         \end{array}
                       \right)^{-1} \left(
                                           \begin{array}{c}
                                             x^k_1+x^k_2\\
                                             x^k_1 x^k_2+x^k_2-x^k_1  \\
                                           \end{array}
                                         \right)   =\left(
                                           \begin{array}{c}
                                             0\\
                                             x^k_1 x^k_2
                                           \end{array}
                                         \right).
$$
Moreover,
$
     \|x^{k+1}-0\|\leq c\|x^{k}-0\|^2.
$
\end{example}

\bigskip

\begin{example}[\cite{SzPrTr12}]
\label{5ex11}
 Consider the following problem
$$
 \min_{x\in \mathbb{R}^2} \, f(x),
$$
where $f:\mathbb{R}^2\rightarrow \mathbb{R}$ is defined by $f(x)=x_1^2+x_1^2x_2+x_2^4$. Let
$
F(x)=f'(x)$. Then $f'(x)=\left(%
\begin{array}{c}
  2x_1+2x_1 x_2 \\
  x_1^2+4x_2^3 \\
\end{array}%
\right)$ and
$\bar{x}=(0,0)^T$.
It is easy to see that $F$ is 3-regular at $\bar{x}$ along $h=(1,1)^T$.

In this example, $$\bar{P}_1=\left(%
\begin{array}{cc}
  0 & 0 \\
  0 & 1 \\
\end{array}%
\right), \quad
\bar{P}_2=\dfrac{1}{2}\left(%
\begin{array}{cc}
  1 & -1 \\
  -1 & 1 \\
\end{array}%
\right),$$
$$P_1=\bar{P}_1+\bar{P}_2=\dfrac{1}{2}\left(%
\begin{array}{cc}
  1 & -1 \\
  -1 & 3 \\
\end{array}%
\right), \quad
P_2=\bar{P}_2\bar{P}_1=\dfrac{1}{2}\left(%
\begin{array}{cc}
  0 & -1 \\
  0 & 1 \\
\end{array}%
\right).$$
Then the following matrix is nonsingular:
$$
  F'(0)+P_1 F''(0)h+P_2 F^{(3)}(0)[h]^2 = f''(0)+P_1 f^{(3)}(0)h+P_2 f^{(4)}(0)[h]^2=
 \left(%
\begin{array}{cc}
  2 & -11 \\
  2 & 11 \\
\end{array}%
\right).
$$

Consider the 3-factor method:
$$  
\begin{array}{l}
x^{k+1} = x^k-\\
\left(f''(0)+P_1f^{(3)}(0)[h]+P_2f^{(4)}(0)[h]^2\right)^{-1}
\left(f'(x^k)+P_1f''(x^k)[h]+P_2f^{(3)}(x^k)[h]^2\right).
\end {array}
$$

Let  $x^k=(x_1,x_2)^T.$ Then
\begin{eqnarray*}
\|x^{k+1}-0\| &=& \left\|x^k-\left(%
\begin{array}{cc}
  2 & -11 \\
  2 & 11 \\
\end{array}%
\right)^{-1}\left(%
\begin{array}{c}
  2x_1-11x_2+2x_1 x_2-6x_2^2 \\
  2x_1+11x_2+x_1^2+18x_2^2+4x_2^3 \\
\end{array}%
\right) \right\|=\\
&=&\dfrac{1}{44}\left\|\left(%
\begin{array}{c}
  11x_1^2+132x_2^2+22x_1 x_2+44x_2^3 \\
  2x_1^2+48x_2^2-4x_1x_2+8x_2^3 \\
\end{array}%
\right) \right\|\leq 10 \|x^{k}-0\|^2.
\end{eqnarray*}
\end{example}

\subsection{Optimality conditions for equality-constrained optimization problems}
\label{sub3.2}
In this section, we consider optimization problem \eqref{phi-min}:
$$
\min f(x) \quad \hbox{ subject to } \quad F(x)=0,
$$
where $f:X\rightarrow\mathbb{R}$  is a sufficiently smooth function and  $F:X\rightarrow Y$ a sufficiently smooth mapping from a Banach space $X$ to a
Banach space $Y$.

\subsubsection{Optimality conditions. Lagrange multiplier theorem}
\label{sub451}

There is an extensive body of literature discussing optimality conditions for regular constrained optimization problems, which are problems that satisfy certain constraint qualifications. One notable reference on this topic is Chapter 3 of the book~\cite{BoSh2000}.

The classical optimality  conditions state that if $\bar{x}$ is a regular solution of Problem \eqref{phi-min}, then there exists Lagrange multiplier $\bar{\lambda}\in Y^{\ast}$ such that
\begin{equation} \label{Opt=}
f '(\bar{x})=F'(\bar{x})^{\ast}\bar{\lambda}.
\end{equation}

The situation changes  in the degenerate case when the derivative $F\,'(\bar{x})$ is not onto, so classical optimality  conditions in the form of equation \eqref{Opt=} do not  hold. 
As an example, consider the following problem 
\begin{equation} \label{5eq4}
   \begin{array} {ll}
       \mathop{\rm minimize}\limits_{x \in \mathbb{R}^{3}} &
        x_2^2+x_3 \\
       {\rm subject \; to}\; & \left( 
       \begin{array}{c}
                                                                            x_1^2-x_2^2+x_3^2 \\
                                                                            x_1^2-x_2^2+x_3^2+x_2x_3 
                                                                            \end{array}%
                                                                            \right) = \left(
                                                                            \begin{array}{c}
                                                                            0 \\
                                                                            0 
                                                                            \end{array}%
                                                                            \right).
   \end{array}
\end{equation}
Note that mapping $F(x)=\left(%
                                                                            \begin{array}{c}
                                                                            x_1^2-x_2^2+x_3^2 \\
                                                                            x_1^2-x_2^2+x_3^2+x_2x_3 \\
                                                                            \end{array}%
                                                                            \right)$  was  introduced in \eqref{F=exx}
In this problem,  
$$\bar{x}=(0,0,0)^T, \quad 
f'(\bar{x})=(0,0,1)^T, \quad
\mbox{and} \quad
F'(\bar{x})=\left(%
\begin{array}{ccc}
  0 & 0 & 0 \\
  0 & 0 & 0 \\
\end{array}%
\right).$$ Hence
$f'(\bar{x})\neq F'(\bar{x})^T \bar{\lambda}$ and equation \eqref{Opt=} does not hold.

\subsubsection{Optimality conditions for $p$-regular optimization problems}
\label{sub5.4}

In this section, we will focus on the case when the equality constraints defined by  mapping $F(x)$
are not regular at a solution $\bar{x}$ of the problem (\ref{phi-min}).
We define the $p$-factor-Lagrange function as
$$\mathcal{L}_p(x,\lambda (h),h)=  f(x)+  \mathop{\sum}\limits_{i=1}^{p} 
      \langle \lambda_i(h), \, f_i^{(i-1)}(x)[h]^{i-1}
      \rangle,$$
where $x \in X$, $h \in X$, 
$\lambda_i(h) \in Y_i^*$, $i=1,\dots,p$, and mappings $f_i(x)$ are defined in \eqref{4eq2}.
Note that the $p$-factor-Lagrange function is a generalization of the Lagrange function 
and it reduces to the classical Lagrange function in the regular case.

The development of optimality conditions for nonregular problems has become an active research topic (see \cite{Av2013, BrTr03, BrTr10, CON2021, Gfrerer, Tr84} and references therein).
      
To state sufficient conditions in Theorem \ref{5th7}, we also need another version of the $p$-factor-Lagrange function that is defined as follows:
\begin{equation}  \label{Lbar}
  \bar{\mathcal{L}}_p(x,\lambda (h),h) =   f(x)+  \mathop{\sum}\limits_{i=1}^{p}  \dfrac{2}{i(i+1)}
      \langle \lambda_i(h), \, f_i^{(i-1)}(x)[h]^{i-1}
      \rangle.
\end{equation}

To state optimality conditions for $p$-regular optimization problems, we use \ref{5def2} of a strongly $p$-regular at $\bar{x}$  mapping $F$. Recall that  set $H_p(\bar{x})$ is defined in Equation \eqref{Hp} and
operator $\Psi_p(h )$ is defined in Equation \eqref{4eq3}.

\begin{theorem}[\cite{Tr84}, Necessary and sufficient conditions for optimality]
\label{5th7}
Let $X$ and $Y$ be Banach spaces, $U$ be a neighborhood of a point $\bar{x}$ in $X$,
$f:U\rightarrow \mathbb{R}$ be a twice continuously Fr\'echet differentiable function in $U$, and 
$F: U \to Y$ be a $(p+1)$-times Fr\'echet differentiable mapping in $U$.

\vspace{2mm}

{\bf Necessary conditions for optimality.}

Assume that for an element $h \in H_p(\bar{x})$  the set $Im\, \Psi_p(h )$ is
closed in $Y_1 \oplus \ldots \oplus Y_p$.
Suppose that $F$ is $p$-regular at the point
$\bar{x}$ along the vector $h\in H_p(\bar{x})$. If $\bar{x}$ is a local minimizer to problem
\ref{phi-min}, then there exist multipliers 
$\bar{\lambda}(h)\in Y^{\ast}$ such that
\begin{equation}
\label{5eq2}
    \mathcal{L}'_{px}(\bar{x},\bar{\lambda}(h),h)=0.
   \end{equation}
   
{\bf Sufficient conditions for optimality.}   
Assume that the set ${\rm Im} \Psi_p(h )$ is closed in
$Y_1 \oplus \ldots \oplus Y_p$ for any element $h \in H_p(\bar{x})$
and
${\rm Im} \, \Psi_p(h) = Y_1 \oplus \ldots \oplus Y_p.$ 
Assume that mapping $F$ is strongly $p$-regular at $\bar{x},$ there
exist constant $\alpha>0$ and a multiplier $\bar{\lambda}(h)$ such that
\ref{5eq2} is satisfied and
\begin{equation}
\label{5eq3}
    \bar{\mathcal{L}}_{pxx}(\bar{x},\bar{\lambda}(h),h)[h]^2\geq \alpha \|h\|^{2}
\end{equation}
for every $h\in H_p(\bar{x}),$ where $ \bar{\mathcal{L}}_{pxx}(x,\lambda (h),h)$ is defined in Equation \eqref{Lbar}.
Then $\bar{x}$ is a strict local
minimizer to the problem \ref{phi-min}.
\end{theorem}

\begin{example}
\label{5ex8}

 Continue consideration of problem \eqref{5eq4} from \ref{sub451}.
 It is easy to verify that the point $\bar{x}=0$ is a local minimum to
\ref{5eq4}.

We have shown in  \ref{5ex4} that mapping $F(x)=\left(%
                                                                            \begin{array}{c}
                                                                            x_1^2-x_2^2+x_3^2 \\
                                                                            x_1^2-x_2^2+x_3^2+x_2x_3 \\
                                                                            \end{array}%
                                                                            \right)$ is $2$-regular at $\bar{x}$ along vector $h=(1,1,0)^T$. 

In this example,  the
$2$-factor-Lagrange function  with $\lambda(h)=(\lambda_1(h),\lambda_2(h))$,
$$\mathcal{L}_2(x,\lambda (h),h)=  f(x)+  \langle \lambda_1(h), \, f_1(x) \rangle +    \langle \lambda_2(h), \, f_2^\prime (x)[h]
      \rangle,$$
 has the form
 $$
    \mathcal{L}_2(x,\lambda(h),h)=x_2^2+x_3+\alpha(x_1-x_2)+\beta(x_1-x_2+\dfrac12x_3),
   $$
   where    $\lambda_1(h)=(0,0)^T$,  and $\lambda_2(h)=(\alpha,\beta)^T.$

Solving the equation $$\mathcal{L}'_{2\,x}(\bar{x},\lambda(h),h)=0,$$
we obtain the following system:
\begin{eqnarray*}
\alpha+\beta& = & 0\\
2x_2^* - \alpha -\beta& = & 0\\
1+ \dfrac12 \beta& = & 0
\end{eqnarray*}
Then with $\bar{x}_2 =0$, we get $\beta=-2$ and $\alpha = 2$. Hence, function $  \bar{\mathcal{L}}_2(x,\lambda (h),h) $ defined in \eqref{Lbar} will have the following form in this example:
$$\bar{\mathcal{L}}_2(x,\lambda (h),h) = x_2^2+x_3 +  \dfrac23 (x_1-x_2) -\dfrac23(x_1-x_2+\dfrac12x_3)= x_2^2+\dfrac23 x_3 .$$

Recall that the set $H_2(\bar{x})$ was determined in \ref{5ex4} and is defined as 
$$H_2(\bar{x}) = \operatorname{span}\left( \left[ \begin{array}{c} 1\\ -1 \\0 \end{array} \right] \right) \bigcup \operatorname{span}\left( \left[ \begin{array}{c} 1\\ 1 \\0 \end{array} \right] \right).$$
Then
 $$\bar{\mathcal{L}}''_{2\,xx}(\bar{x},\lambda(h),h)[h]^2=2\geq \alpha \|h\|^{2}$$
 for some $\alpha>0$ and 
for every $h\in H_2(\bar{x})$.
Hence, the sufficient conditions in Theorem \ref{5th7} hold and we conclude that $\bar{x}$ is a strict local minimizer in problem \ref{5eq4}.
\end{example}

\subsection{Modified Lagrange function method}
\label{sub3.6}

\subsubsection{Modified Lagrange function method}
\label{sub3.61}

Consider the following constrained optimization problem
\begin{equation}
\label{3eq9}
\min f(x) \quad  \hbox{subject to } \quad g_i(x)\leq 0, \quad i=1,\ldots,m,
\end{equation}
where $f:\mathbb{R}^n\rightarrow \mathbb{R}$, $g_i:\mathbb{R}^n\rightarrow \mathbb{R}$
and the modified Lagrangian function $L_E(x,\lambda)$, $L_E:\mathbb{R}^{n+m}\rightarrow \mathbb{R}$ associated with problem \eqref{3eq9} is defined as follows (see e.g. \cite{BrEvTr06,Ev77}, cf.  \cite{Be16}):
\begin{equation}
\label{LE}
L_E(x,\lambda)=f(x)+\dfrac{1}{2}\sum_{i=1}^m \lambda_i^2g_i(x),
\end{equation}
where $\lambda = ( \lambda_1, \ldots,  \lambda_m)$.
This modified Lagrangian function  is used to replace the nonlinear optimization problem with a system of nonlinear equations.
Define mapping $G:\mathbb{R}^n\times \mathbb{R}^m\rightarrow \mathbb{R}^{n+m}$ by 
\begin{equation}
\label{3eq10}
G(x,\lambda)=\left(%
\begin{array}{c}
\nabla f(x)+ \dfrac{1}{2}\sum\limits_{i=1}^m \lambda_i^2 \nabla g_i(x)\\
D(\lambda) g(x) \\
\end{array}%
\right),
\end{equation}
where $D(\lambda)= \hbox{diag}\{\lambda_i\},$ $i=1,\ldots,m$, $\lambda\in \mathbb{R}^m$.

Consider the equation
\begin{equation}
\label{3eq11}
G(x,\lambda)=0_{n+m}.
\end{equation}
Let $g^\prime (x)$ be the Jacobi matrix of the mapping $g(x)$. 
Then the  Jacobian matrix $G'(x,\lambda)$ of mapping $G(x,\lambda)$ is given by
\begin{equation}
\label{MG}
G'(x,\lambda)=\left(%
\begin{array}{cc}
\nabla^2 f(x)+ \dfrac{1}{2}\sum\limits_{i=1}^m \lambda_i^2 g_i^{\prime \prime}(x)&(g'(x))^T D(\lambda) \\
D(\lambda) (g^\prime(x))^T& D(g(x)) \\
\end{array}%
\right).
\end{equation}

  Define the set $I(\bar{x})=\{j=1,\ldots, m \mid g_j(\bar{x})=0\}$ of active constraints, the set $$I_0(\bar{x})=\{j=1,\ldots, m \mid \bar{\lambda}_j=0, g_j(\bar{x})=0\}$$ of weakly active constraints, and the set $I_+(\bar{x})=I(\bar{x})\setminus I_0(\bar{x})$ of strongly active constraints.

If $(\bar{x},\bar{\lambda})$ is a solution  of Problem \eqref{3eq11} such that $g_i(\bar{x})=0$ and
$\bar{\lambda}_i=0$, then the strict complementarity condition (SCQ) fails, that is the set
$I_0(\bar{x})$
is not empty. 
Consequently, $G'(\bar{x},\bar{\lambda})$  is a degenerate matrix. Example \ref{3ex4} illustrates the situation.

\begin{example}\cite{BrEvTr06}
\label{3ex4}
Consider the following problem
\begin{equation}
\label{3eq12}
\min_{x\in \mathbb{R}^n} (x_1^2+x_2^2+4 x_1 x_2) \quad  \hbox{subject to } \quad x_1\geq 0,\; x_2\geq 0.
\end{equation}
It is easy to see that $\bar{x}=(0,0)^T$ is a solution of Problem \eqref{3eq12} with the corresponding Lagrange multiplier $\bar{\lambda}=(0,0)^T$.

The modified Lagrange function in this case is
$$L_E (x,\lambda) =x_1^2+x_2^2+4x_1x_2-\dfrac{1}{2}\lambda_1^2x_1-\dfrac{1}{2}\lambda_2^2x_2.$$
The mapping $G$ is  defined by
$$G(x,\lambda)=\left(%
\begin{array}{c}
2x_1+4x_2-\dfrac{1}{2}\lambda_1^2\\
2x_2+4x_1-\dfrac{1}{2}\lambda_2^2 \\
-\lambda_1x_1\\
-\lambda_2x_2\\
\end{array}%
\right),$$ and, therefore,  the Jacobian matrix $G'(\bar{x},\bar{\lambda})$  defined in \eqref{MG} 
 is singular.
\end{example}

\subsubsection{Modified Lagrange function method for 2-regular problems}
\label{sub5.7}

In this section, we consider  the constrained optimization problem  \ref{3eq9} with the modified Lagrangian function $L_E(x,\lambda)$ defined in \eqref{LE}.
We focus on the nonregular case when the Jacobian matrix $G'(\bar{x},\bar{\lambda})$  
defined in \eqref{MG}  is singular at the solution $(\bar{x},\bar{\lambda})$ of \ref{3eq11}.
 
 We will show that the mapping $G(x,\lambda)$ defined by \ref{3eq10} is 2-regular at $(\bar{x},\bar{\lambda})$.
 
Without loss of generality, assume that $I_0(\bar{x})=\{1,\ldots, s\}$, so that $\bar{\lambda}_j=0$ and $g_j(\bar{x})=0$ for all $j=1,\ldots, s$. 
Additionally, we assume that $I_+(\bar{x}) = \{s+1, s+2, \ldots, p\}$. Introduce the notation $l = m-p$.
Then the rows of matrix $G'(\bar{x},\bar{\lambda})$ with the numbers from $(n+1)$th to $(n+s)$th contain only zeros. 
Define vector $h\in \mathbb{R}^{n+m}$ as follows
\begin{equation}
\label{5eq12}
h^T=\left(0_n^T,1_s^T,0_{m-s}^T\right),
\end{equation}
where $1_s^T$ is an $s$-dimensional  all-one row vector. 

Let mapping $\Phi:\mathbb{R}^n\times \mathbb{R}^m$ be given by
\begin{equation}
\label{5eq13}
\Phi(x,\lambda)=G(x,\lambda)+G'(x,\lambda)h,
\end{equation}
with $h$ is defined in \ref{5eq12}.

The following is a well known result.
\begin{lemma}[\cite{BrEvTr06}]
\label{5lem14}
	Let an $n\times n$ matrix $V$ and an $n\times p$ matrix $Q$ satisfy the properties:
	\begin{enumerate}
	\item $Q$ has linearly independent columns and
	\item
	$x^T Vx >0$ for all ${x\in \Ker   Q^T\setminus\{0\}}$.
	\end{enumerate}
	Assume also that $D_{N}$ is a full rank diagonal $l\times l$ matrix. Then matrix $\bar{A}$ defined by
\begin{equation} \label{barA}
\bar{A}=\left(
	\begin{array}{ccc}
	V & Q & 0 \\
	Q^T & 0 & 0 \\
	0 & 0 & D_N \\
	\end{array}
	\right)
\end{equation}
	is a nonsingular matrix.
\end{lemma}

We say that the linear independence constraint qualification (LICQ) holds for the optimization problem  \ref{3eq9} if the gradients of active constraints are linearly independent. 

The second-order sufficient optimality condition holds if there exists $\alpha >0$ such that
\begin{equation}
\label{5eq14}
  z^T \nabla_{xx}^2L_E(\bar{x},\bar{\lambda})z \geq \alpha \|z\|^2
\end{equation}
for all $z\in \mathbb{R}^n$ satisfying the conditions
$$(\nabla g_j(\bar{x}))^T z\leq 0 \quad \forall j\in I(\bar{x}).$$

\begin{lemma}[\cite{BrEvTr06}]
\label{5lem15}
	Let $f, g_i\in \mathcal{C}^3(\mathbb{R}^n)$, $i=1,\ldots,m$. 
	Assume that the LICQ and the second-order sufficient optimality conditions are satisfied at the solution $(\bar{x},\bar{\lambda})$ and $\Phi$ is a mapping given by \ref{5eq13}. 
	Then the 2-factor operator $$ \Phi'(x,\lambda)=G'(x,\lambda)+G''(x,\lambda)h$$ is nonsingular at the point $(\bar{x},\bar{\lambda})$.
\end{lemma}

The proof of \ref{5lem15} can be obtained from  \ref{5lem14}.

Indeed, if matrix $D(\lambda)$ is a diagonal matrix with $\lambda_j$ being the $j$-th diagonal entry,
$$V=\nabla_{xx}^2L_E(\bar{x},\bar{\lambda}), \quad D_N=D(g_N(\bar{x})), \quad g_N(x)=\left(g_{p+1}(x),\ldots,g_m(x)\right)^T,
$$
and
$$
Q=\left[\nabla g_1(\bar{x}), \ldots, \nabla g_s(\bar{x}), \bar{\lambda}_{s+1}\nabla g_{s+1}(\bar{x}),\ldots,  \bar{\lambda}_{p}\nabla g_{p}(\bar{x})\right],$$
then $\Phi'(\bar{x},\bar{\lambda})=\bar{A}$, where matrix $\bar{A}$ is defined in \eqref{barA}.
\ref{5lem15} implies that the 2-factor-Newton method given by
\begin{equation}
\label{5eq15}
w^{k+1}=w^k-\left(G'(w^k)+G''(w^k)h\right)^{-1}\left(G(w^k)+G'(w^k)h\right), \quad k=0,1,\ldots
\end{equation}
can be applied to solve system \ref{3eq11}. As a result, we have the following proposition.

\begin{theorem}[\cite{BrEvTr06}]
\label{5th16}
	Let $\bar{x}$ be a solution to \ref{3eq9} and $f, g_i(x)\in \mathcal{C}^3(\mathbb{R}^n)$, $i=1,\ldots,m$. Assume that the  LICQ and the second-order sufficient optimality conditions \ref{5eq14} are satisfied at the point $\bar{x}$. Then there exists a sufficiently small neighborhood $B_{\varepsilon}(\bar{w})$ of $\bar{w}=(\bar{x},\bar{\lambda})$ such that  the estimate
	$$\|w^{k+1}-\bar{w}\|\leq \beta \|w^k-\bar{w}\|^2,$$
holds for the method \ref{5eq15}, where $w^0\in B_{\varepsilon}(\bar{w})$ and $\beta>0$ is an independent constant.
\end{theorem}

There are publications of other authors where a modified Lagrange function is used in various contexts, for example,
\cite{Antipin, Zhil}.

\subsection{Calculus of variations}
\label{sub3.5}

The methods of the calculus of variations are used as a tool in solving many problems in physics as well as in classical mechanics. Since the classical approach cannot be directly applied to many of these problems, there is a need to extend or reformulate the classical theorems for some irregular cases. Different kinds of irregular problems of calculus of variations have been studied widely over the years in mathematics and applications (see e.g. \cite{BradEver1973,MR4393590,Io2000,ioffe,MR3727108,KonKunOber2008,LecBagGio2019,GelFom1963,Sival1992,Tuckey1993}).

\subsubsection{Singular problems of calculus of variations}
In this section we consider the following Lagrange problem (see \cite{PrSzTr13})
\begin{equation}\label{3eq3}
J_{0}(x)=\int_{t_1} ^{t_2} f(t,x,x^{\prime}) \mathrm{d}t\rightarrow \min
\end{equation}
subject to the subsidiary conditions
\begin{equation}\label{3eq4}
G(x)=\tilde{G}(t,x,x^{\prime})=0, \;\;\;  q(x(t_1),x(t_2))=0 ,
\end{equation}
where 
$G:X\rightarrow Y,$ $G\in \mathcal{C}^{p+1}(X),$ $X=\mathcal{C}^1([t_1,t_2],\mathbb{R}^n),$ $x \in  X$, 
$Y=\mathcal{C}^1([t_1,t_2],\mathbb{R}^m), $
$\tilde{G}(t,x,x^{\prime})=(G_1(t,x,x^{\prime}),\ldots,G_m(t,x,x^{\prime})),$ and
$q: \mathbb{R}^n \times \mathbb{R}^n\rightarrow
\mathbb{R}^k .$  We assume that all mappings and their
derivatives are continuous with respect to the corresponding variables
$ t,$ $x,$ and $x'$.

We denote by $\bar{x}(t)$ a solution to Problem \eqref{3eq3}--\eqref{3eq4}.
In the regular case, when $ \operatorname{Im}  G^{\prime}(\bar{x})=Y,$ the
Euler-Lagrange necessary conditions are satisfied and have the form (see e.g. \cite{Bellman1985,GelFom1963}):
\begin{equation}
\label{3eq5}
f_x+\lambda(t)G_x-\dfrac{d}{dt}(f_{x^{\prime}}+\lambda(t)G_{x^{\prime}})=0 \quad \hbox{for all } t\in [t_1,t_2].
\end{equation}

Let $\lambda=(\lambda_1,\ldots,\lambda_m)^T,\ $ then
$$\lambda(t)G=\lambda_1(t)G_1+\cdots+\lambda_m(t)G_m \quad \mbox{and} \quad 
\lambda(t)G_x=\lambda_1(t)G_{1x}+\cdots+\lambda_m(t)G_{mx}.$$
In the singular case, when
$
\operatorname{Im}  G'(\bar{x})\neq Y,
$
we can only guarantee that the following equations are satisfied:
\begin{equation}
\label{3eq6}
\lambda_0 f_x+\lambda(t)G_x-\dfrac{d}{dt}(\lambda_0f_{x'}+\lambda(t)G_{x'})=0 ,
\end{equation}
where $\lambda_0^2+\|\lambda (t)\|^2=1.$ In this case, $\lambda_0$
might be equal to $0$, which gives no 
constructive conditions for description or finding $\bar{x}(t)$.

\begin{example}[\cite{PrSzTr13}]
\label{3ex3}
	Consider the following simple problem
	\begin{equation}
\label{3eq7}
	J_{0}(x)=\int_{0} ^{2\pi} (x_1^2+x_2^2+x_3^2+x_4^2+x_5^2)
	\mathrm{d}t\rightarrow \min
	\end{equation}
	subject to
	\begin{equation}
\label{3eq8}
	G(x)=\left( \begin{array}{c}
	x_1^{\prime}-x_2+x_3^2x_1+x_4^2x_2-x_5^2(x_1+x_2) \\
	x_2^{\prime}+x_1+x_3^2x_2-x_4^2x_1-x_5^2(x_2-x_1) \\
\end{array}
	\right)=0,
	\end{equation}
	$$x_i(0)=x_i(2\pi), \quad i=1,\ldots,5. $$ Here  $
	f(x)=x_1^2+x_2^2+x_3^2+x_4^2+x_5^2$ and  $q_i(x(0),x(2\pi))=x_i(0)-x_i(2\pi),$ $i=1,\ldots,5.$
	The solution of Problem \eqref{3eq7}--\eqref{3eq8} is $\bar{x}(t)=0$ and
	$G^{\prime}(\bar{x}(t))$ is singular.
	Indeed, $$
	G^{\prime}(0)=\left(%
	\begin{array}{c}
	(\cdot)^{\prime}_1-(\cdot)_2 \\[2mm]
	(\cdot)^{\prime}_2+(\cdot)_1 \\
	\end{array}%
	\right)
\quad
	\mbox{and}
\quad
	G^{\prime}(0)x=\left(%
	\begin{array}{c}
	x^{\prime}_1-x_2 \\
	x^{\prime}_2+x_1 \\
	\end{array}%
	\right).
	$$
	
	Let $z(t)=x_1(t)$ and consider the following equivalent
	problem. Determine whether the following mapping is surjective or not: $$ \tilde{G}(z)=z^{\prime\prime}+z, \;\;\;z(0)=z(2\pi).$$
	
	It is obvious that for $y\in \mathcal{C}[0,2\pi],$ satisfying 
	$$ \int_{0} ^{2\pi}\sin \tau\ y(\tau)\mathrm{d}\tau\neq 0 \; \hbox{ or }\; \int_{0} ^{2\pi}\cos \tau\ y(\tau)\mathrm{d}\tau\neq 0,$$ the equation
	$z^{\prime\prime}+z=y$ does not have a solution.
	
	The corresponding Euler-Lagrange equations in
	this case are as follows:
	\begin{eqnarray*}\label{eq20}
		2\lambda_0x_1+\lambda_2-\lambda'_1+\lambda_1x_3^2+\lambda_2 x_5^2-\lambda_2 x_4^2 &=& 0, \\
		2\lambda_0x_2-\lambda_1-\lambda'_2+\lambda_1x_4^2+\lambda_2x_3^2-\lambda_1x_5^2-\lambda_2x_5^2 &=& 0, \\
		2\lambda_0x_3+2\lambda_1x_1x_3+2\lambda_2x_2x_3 &=& 0, \\
		2\lambda_0x_4+2\lambda_1x_2x_4- \lambda_2x_1x_4 &=& 0, \\
		2\lambda_0x_5-2\lambda_1x_5x_1- 2\lambda_1x_2x_5-2\lambda_2x_2x_5+2\lambda_2x_1x_5 &=& 0, \\
		\lambda_i(0)=\lambda_i(2\pi),\;\; i=1,2.& &
	\end{eqnarray*}
	Unfortunately, we cannot guarantee that $\lambda_0\neq 0$. For
	$\lambda_0 = 0$, we obtain a series of spurious
	solutions to the system \eqref{3eq7}--\eqref{3eq8}:
	$$ x_1=a\sin t, \ x_2=a\cos t, \quad x_3= x_4=x_5=0,\quad \lambda_1=b\sin t,\quad \lambda_2=b\cos t,\quad a,b \in
	\mathbb{R}.$$
\end{example}

\subsubsection{Optimality conditions for $p$-regular problems of calculus of variations}
\label{sub5.6}

To formulate optimality conditions for singular problems of the form \ref{3eq3}--\ref{3eq4}
 we define the $p$-factor Euler-Lagrange function by
$$S(x) =f(x)+\lambda(t)G^{(p-1)}(x)[h]^{p-1},$$
where\\
$$G^{(p-1)}(x)[h]^{p-1}=g_1(x)+g'_2(x)[h]+\ldots +g_p^{(p-1)}(x)[h]^{p-1} ,$$
$$\lambda(t)G^{(p-1)}(x)[h]^{p-1}=\left\langle \lambda(t),\left(g_1(x)+g'_2(x)[h]+\ldots+g_p^{(p-1)}(x)[h]^{p-1}\right)\right\rangle,$$
$$\lambda(t)=(\lambda_1(t),\ldots,\lambda_m(t))^{T}.$$ 

Functions $g_i(x),$ $i=1,\ldots,p$, are determined for the mapping $G(x)$ in a way that is similar to how functions $f_i(x),$ $i=1,\ldots,p,$ 
are defined for the mapping $F(x),$ in \ref{4eq2}. Namely, $$g_k(x)=P_{Y_k}G(x), \quad k=1,\ldots,p.$$

Let
$$g_k^{(k-1)}(x)[h]^{k-1}= \sum_{i+j=k-1}C_{k-1}^i {g_k^{(k-1)}}_{x^i(x')^j}(x)[h]^i[h']^{j},\quad k=1,\ldots,p,$$
where
${g_k^{(k-1)}}_{x^i(x')^{j}}(x)={g_k^{(k-1)}}_{\underbrace{x\ldots x}_i \underbrace{x'\ldots x'}_j}(x).$

\begin{definition}
\label{5def12}
We say that problem \ref{3eq3}--\ref{3eq4} is $p$-regular at the point $\bar{x}$ along some vector $h\in \bigcap\limits_{k=1}^{p}\Ker^k g_k^{(k)}(\bar{x}),$ $\|h\|\neq 0$, if $$\operatorname{Im}  \left(g'_1(\bar{x})+\ldots+g_p^{(p)}(\bar{x})[h]^{p-1}\right)= \mathcal{C}_m[t_1,t_2].$$
\end{definition}

\begin{theorem}[\cite{PrSzTr13}]
\label{5th13}
Let $\bar{x}(t)$ be a solution of problem \ref{3eq3}--\ref{3eq4}, which is $p$-regular at $\bar{x}$ along
$h\in \bigcap\limits_{k=1}^{p}\Ker^k g_k^{(k)}(\bar{x}).$ Then there exists a multiplier $\hat{\lambda}(t)=(\hat{\lambda}_1(t),\ldots,\hat{\lambda}_m(t))^{T}$ such that the following $p$-factor Euler-Lagrange equation holds:
\begin{equation}
\label{5eq10}
\begin{array} {l}
  S_x(\bar{x}) -\dfrac{d}{dt}S_{x'}(\bar{x})=f_x(\bar{x})+\left\langle\hat{\lambda},\sum_{k=1}^{p}\sum_{i+j=k-1}C_{k-1}^i g_{x^i(x')^j}^{(k-1)}(\bar{x})[h]^i(h')^{j}\right\rangle_{x}- \\[4mm]
 \hspace*{15mm} -\dfrac{d}{dt}\left[f_{x'}(\bar{x})+\left\langle \hat{\lambda}(t),\sum_{k=1}^{p}\sum_{i+j=k-1}C_{k-1}^i g_{x^i(x')^j}^{(k-1)}(\bar{x})[h]^i(h')^{j}\right\rangle_{x'}\right] =0,\\[4mm]
 \hspace*{15mm}  \lambda_i(0)=\lambda_i(2\pi),\;\; i=1,2. 
\end{array}
\end{equation}
\end{theorem}

The proof of the above theorem is similar to the one for the singular isoperimetric problem in \cite{BeTr88} or \cite{KoTr02}.

\smallskip

Consider \ref{3ex3}. 

The mapping $G$ is $2$-regular at the point $\bar{x}=(a\sin t, a \cos t, 0,0,0)^T$  along $h=(a\sin t, a \cos t, 1,1,1)^T$. It  means that in that problem $p$ is equal to $2$.

\smallskip

Consider the following equation
$$f_x(x)+(G'(x)+P_{Y_2}G''(x)h)^{\ast}\lambda=0 ,$$
which is equivalent to the system of equations
\begin{equation}
\label{5eq11}
    \left\{
                    \begin{array}{l}
                      2x_1-\lambda'_1+\lambda_2=0 \\
                      2x_2- \lambda'_2-\lambda_1=0\\
                      2 x_3+2\lambda_1a\sin t+2\lambda_2 a \cos t=0\\
                      2x_4+2 \lambda_1 a \cos t-2\lambda_2a \sin t=0\\
                      2 x_5+2\lambda_1 a(\cos t-\sin t)+2\lambda_2 a(\sin t-\cos t)=0.\\
                      \lambda_i(0)=\lambda_i(2\pi), \; i=1,2.
                    \end{array}
                  \right.
\end{equation}

One can verify that the following ``false solutions" of \ref{3eq7}--\ref{3eq8},
$$x_1=a\sin t, \; x_2=a\cos t, \;x_3=x_4=x_5=0,$$
do not satisfy the system \ref{5eq11} if $a\neq 0.$ It means that $$x_1=a\sin t, \quad x_2=a\cos t, \quad x_3=x_4=x_5$$ do not satisfy the 2-factor Euler-Lagrange equation \ref{5eq10} from \ref{5th13}. The only solution to \ref{3ex3} is $\bar{x}=(0,0,0,0,0)^T.$
Indeed, the 2-factor Euler-Lagrange equation in this case has the following form for $\bar{x}=(0,0,0,0,0)^T$:
\begin{equation*}
     \left\{
                    \begin{array}{l}
                      -\lambda'_1+\lambda_2=0 \\
                      - \lambda'_2-\lambda_1=0\\
                      2\lambda_1a\sin t+2\lambda_2 a \cos t=0\\
                      2 \lambda_1 a \cos t-2\lambda_2a \sin t=0\\
                      2\lambda_1 a(\cos t-\sin t)+2\lambda_2 a(\sin t-\cos t)=0.\\
                      \lambda_i(0)=\lambda_i(\pi), \; i=1,2,\\
                    \end{array}
                  \right.
\end{equation*}
where the solution is $\bar{\lambda}_i(t)=0$, $i=1,2.$

\subsection{Existence of solutions to nonlinear equations}
\label{sec:Ex}

This section deals with the existence of a solution to the equation of the form $F(x)=0$ in the neighbourhood of a chosen point. A very general situation is considered when the function $F$ acts from Banach space $X$ to Banach space $Y$, and the assumptions concern the properties of its derivatives in the considered neighbourhood. Thus, it is one of the classic problems of nonlinear analysis, with many important applications, especially in the theory of differential equations (c.f. \cite{CabAleTom2019,Bob1990,ZhaChe2006}).
One of the well-known methods used for this problem is Newton's method (see \cite{Reddien78}). The solution of the equation is obtained as the limit of a recursively defined series of approximations. This method is applied to the proof of the first theorem in this section. In particular, the existence of the inverse operator to the derivative of the function at a chosen point is assumed. The next theorem in this section is more general and uses the $p$-factor construction of the operator for functions of the class $\mathcal{C}^{p+1}$. A certain limitation of this construction is the assumption of the existence of continuous projections on subspaces of $Y$ corresponding to successive orders of derivatives of the function $F$.

\subsubsection{Existence of solutions to nonlinear equations in the regular case}
\label{sub3.8}

 Let $X$ and $Y$ be Banach spaces and let
$B_{\varepsilon}(x^{0})=\{x \in X: \|x-x^{0}\|< \varepsilon\}$ be a ball with the center $x^{0}$ and radius $\varepsilon >0$, where
$0<\varepsilon<1.$ Consider a mapping $F:X\rightarrow Y$ and a problem of 
existence of a point $\bar{x}$ such hat $F(\bar{x})=0.$ We know that
equation $F(x)=0$ is solvable and has a solution $\bar{x}$ when the operator
$F'(x^{0})$ is surjective (\cite{DeMa73,KaAk82}).
A modified version of the following theorem was given in \cite{DeMa73}.

\begin{theorem}
\label{3th5}
Let $F\in \mathcal{C}^{2}(B_{\varepsilon}(x^{0}))$ and $\left\|F(x^{0}) \right\| = \eta$ for some constants $\varepsilon>0$ and $\eta >0$, and some point $x_0 \in X$. Assume that there exist $\left[F'(x^{0})\right]^{-1}$ and constants $\delta >0$ and $C>0$ such that
 $\left\|\left[F'(x^{0})\right]^{-1}\right\|= \delta,$ $\sup\limits_{x \in B_{\varepsilon}(x^{0})} \left\| F''(x)\right\| =
C<+\infty$. If, moreover, the following relations hold:
\begin{enumerate}
\item $\;\;\delta  \eta \leq\dfrac{\varepsilon}{2},$ \item
$\;\;\delta C \varepsilon \leq \dfrac{1}{4},$ \item $\;\;C
 \varepsilon \leq \dfrac{1}{2}$,
\end{enumerate}
then the equation $F(x)=0$ has a solution $\bar{x} \in B_{\varepsilon}(x^{0}).$
\end{theorem}

If  the first derivative of $F$ at $x^{0}$ is not surjective, then Theorem \ref{3th5} cannot be applied. Consider, as an example, a mapping $F:{\mathbb R}\rightarrow {\mathbb R}$ defined by
$$F(x)=\dfrac{1}{7!}x^{7}+x^{5}+\dfrac{1}{10^{3}}.$$ 
Note that if $x^{0}=0$, the assumptions of  \ref{3th5} are not satisfied but the equation $F(x)=0$ has a solution $\bar{x}\approx -0,251188$.

\subsubsection{Existence of solutions to nonlinear equations in singular case}
\label{sub5.9}
In this section, we continue considering the problem stated in Section \ref{sub3.8}. Namely, let $X$ and $Y$ be Banach spaces and $F:X\rightarrow Y$.
Assume that $F(x^{0}) \neq 0$ for some $x^{0}$. We are interested in existence of a solution $\bar{x}$ of the equation $F(x)=0$ in some small neighborhood $B_{\varepsilon}(x^{0})$ of the point $x^{0}$, so that
$F(\bar{x}) =0$. Most of work on solving this problem focuses on Newton's method or some modifications of Newton's method under an assumption that $F'(x^{0})$ is regular (see e.g. \cite{Moore1977}). 

Consider a degenerate case when $F'(x^{0})$ is not regular. The focus will be on finding a small constant $\varepsilon >0$ such that the neighborhood $B_{\varepsilon}(x^{0})$ has a solution $\bar{x}$ of $F(x) =0$.
Introduce the following notations and assumptions for some $p \geq 2$:

\begin{equation} \label{5eq16}
  \delta  = \left\|F(x^{0})\right\|,
\end{equation}
\begin{equation}  \label{5eq18}
  \eta = \left\|\{\Psi_p(h)\}^{-1}\right\|<\infty,\quad   h \in \bigcap_{k=1}^{p}\Ker^{k}f_{k}^{(k)}(x^{0}) , \quad \|h\|=1,
\end{equation}
$$
    c=\max_{k=1,\ldots,p} \sup_{x\in B_{\varepsilon}(x^{0})}\left\|f^{(k+1)}_{k}(x)\right\|,
\quad     d= 4\max_{k=1,\ldots,p}\dfrac{1}{(k-1)!}\left\|f^{(k)}_{k}(x^{0})\right\|,
$$
\begin{equation}  \label{5eq21}
    \alpha= \min\left\{\dfrac{3}{4^{p+2}\eta},\min_{k=1,\ldots,p}\dfrac{\left\|f^{(k)}_{k}(x^{0})\right\|}{(k-1)!}\right\}.
\end{equation}

The following theorem was proved in \cite{PrTr11}.
\begin{theorem}
\label{5th18}
Let $X$ and $Y$ be  Banach spaces and $F: X\rightarrow Y$, $F\in
\mathcal{C}^{p+1}(X).$ Assume that there exists $h\in \mathop{\bigcap}\limits_{k=1}^{p}\Ker^{k}f_{k}^{(k)}(x^{0}),$  $\|h\|=1$, such
that $F$ is a $p$-regular mapping at $\;x^{0} \in X \;$ along
$h.$

Assume also that there exists $\omega,$ $0<\omega< \dfrac{1}{2}\nu,$ $\nu\in (0,1),$ such that the following inequalities hold:
\begin{enumerate}
        \item $ \eta\delta\leq\alpha\dfrac{\omega^{p}}{2pd},$\\
        \item $\dfrac{4^{p+2}}{3}c\omega\eta\leq \dfrac{1}{2}.$\\
\end{enumerate}
Then the equation $F(x)=0$ has a solution $\bar{x} =x^{0}+\omega h+\bar{x}(\omega)\in
B_{\nu}(x^{0})$, where $\bar{x}(\omega)$ is a fixed point such that $\|\bar{x}(\omega)\|\leq\dfrac{1}{2}\omega.$
\end{theorem}

Recall that if our focus is on finding radius $\varepsilon>0$ such that the neighborhood $B_{\varepsilon}(x^{0})$ has a solution $\bar{x}$ of $F(x) =0$,
then Theorem \ref{5th18} implies that $\varepsilon= \omega +\bar{x}(\omega)$. For example, we can take $\varepsilon= \dfrac32 \omega.$

As an example of singular nonlinear equation we consider the problem of 
existence of local nontrivial solutions of the boundary value problem (BVP) for an 
ordinary differential equation 
\begin{equation}
\label{3eq13}
    y''(t) + y(t) + g(y(t))=x(t)
\end{equation}
with the boundary conditions
\begin{equation}
\label{3eq14}
y(0)=y(\pi)=0
\end{equation}
which is degenerate at $y^*(t)=0$.
Here, $y(\cdot)\in \mathcal{C}^{2}([0,\pi])$ and $g, x$ are
given functions such that $$x\in \mathcal{C}[0,\pi], \quad
g\in\mathcal{C}^{p+1}([0,\pi]), \quad
g(\mathbf{0})=g'(\mathbf{0})=\mathbf{0}.$$

The application of Theorem \ref{5th18} is described in more details in Section~\ref{sub3.9}.

\begin{remark}
\label{5rem19}
Recall that the operator $\Psi_p$ is defined in \eqref{4eq3}. 
The surjectivity of operator
 $\Psi_p(\omega h)$ for any $\omega\neq 0$ implies the condition of $p$-regularity of the mapping $F$ at the point $x^{0}$ (by the definition). It is also equivalent to the following inequality with a vector $h$ such that $\|h\|=1$:
 $$\|\{\Psi_p(\omega h)\}^{-1}y\|\leq \left(1+\dfrac{1}{\omega}+\dfrac{1}{w^{2}}+\ldots+\dfrac{1}{\omega^{p-1}} \right).$$
\end{remark}

\subsection{Differential equations}
\label{sub3.9}

\subsubsection{Nonlinear differential equations}

Consider the second-order bo\-un\-da\-ry-value problem with a nonlinear ordinary
differential equation 
\begin{equation}
\label{3eq13s}
    y''(t) + y(t) + g(y(t))=x(t),
\end{equation}
and the boundary conditions
\begin{equation}
\label{3eq14s}
y(0)=y(\pi)=0.
\end{equation}
Here, $y(\cdot)\in \mathcal{C}^{2}([0,\pi])$ and $g, x$ are
given functions such that $$x\in \mathcal{C}[0,\pi], \quad
g\in\mathcal{C}^{p+1}([0,\pi]), \quad
g(\mathbf{0})=g'(\mathbf{0})=\mathbf{0}.$$

We investigate local
existence of nontrivial solutions of equations
\eqref{3eq13s}--\eqref{3eq14s}. Assume that $t\in
[0,\pi].$ Moreover, let $X=\left\{y(\cdot)\in
\mathcal{C}([0,\pi]):y(0)=y(\pi)=0\right\}$, $Y=\left\{y(\cdot)\in
\mathcal{C}^{2}([0,\pi]):y(0)=y(\pi)=0\right\}$ and
$Z=\mathcal{C}([0,\pi])$ be spaces with standard norms. Introduce
the mapping $F:X \times Y {\rightarrow} Z$,
$$
F(x,y)=y''+y+g(y)-x.
$$
Then we can rewrite equation \eqref{3eq13s} as
\begin{equation}
\label{3eq15}
F(x,y)=0.
\end{equation}
Observe that $F(0,0)=0$, i.e. $(0,0)$ is a solution of \eqref{3eq15}. Without loss of generality, we may restrict our attention to some neighborhood $U\times V \subset X\times Y$ of the point $(0,0)$. 
Then the problem of existence  of a solution of the nonhomogeneous BVP \eqref{3eq13s}--\eqref{3eq14s} for $x(t)\in U$ is equivalent to the problem  existence of an implicit function $\phi: U \rightarrow Y$ such that $y=\phi(x)$ and 
\begin{equation}
    \label{7-brez}
    F(x,\phi(x)) =0,\ \ \phi(0)=0.
\end{equation}

In the case when $F(0,0)=0$ and the mapping $F$ is regular at $(0,0)$; i.e., when its derivative with respect to $y$, denoted by $F_y'(0,0)$, is invertible, the classical Implicit Function Theorem guarantees the existence of a smooth mapping $\phi:U\rightarrow Y$ defined in a neighborhood of $\bar{x}=0$ such that $F(x,\phi(x))=0$ and $\phi(0)=0.$

Now, taking $y_{0}(t)\equiv \mathbf{0},$ under above assumptions,
we get that $F(x,y_{0})= - x(t)$ and
$$F'_{y}(0,y_{0})(\cdot)={(\cdot)}''+(\cdot)+g'(y_{0}),$$ so
    $F'_{y}(0,0)\xi={\xi}''+\xi.$ Note, that $F$ is singular at $y_{0}$, {\em i.e.} $F'_{y}(0,y_{0})$ is not onto because  for
$y(t)=\sin t$ the following boundary-value problem
$$
{y}''(t)+y(t)=\sin t, \;\; y(0)=y(\pi)=0
$$
does not have a solution. Hence, the operator $F'_y(0,0)$ is
not surjective (see \cite{BTM08}) and the classical Implicit Function
Theorem cannot be applied to guarantee existence of a solution
of the equation \eqref{3eq15}.

\subsubsection{Nonlinear differential equations in singular case}
\label{sub5.10}

We consider the boundary-value problem \ref{3eq13s}--\ref{3eq14s}
and the mapping
\begin{equation}
\label{5eq22}
   F(x(t),y(t))=y''(t)+y(t)+g(y(t))-x(t), \;\;\;t\in [0,\pi],
\end{equation}
 where $F:X\times Y \rightarrow Z,$ $X=\left\{x(\cdot)\in
\mathcal{C}([0,\pi]):x(0)=x(\pi)=0\right\},$  $Z={C}([0,\pi])$, 
$Y=\left\{y(\cdot)\in
\mathcal{C}^2([0,\pi]):y(0)=y(\pi)=0\right\}$,  and as above, 
\begin{equation} \label{gg0}
g(0)=g'(0)=0.
\end{equation}


We restrict our attention to some neighborhood of the point $(x_0(t), y_0(t))=(0,0)$, $ t\in [0,\pi].$
 Observe that $F'_{y}(0, 0)$ is not surjective.  
$F_y'(0,0)(\cdot)={(\cdot)}''+(\cdot).$ In order to construct the $p$-factor
operator, we first define the image of the operator $F'_{y}(0,0)$. Namely,
$$
\operatorname{Im}  F_y'(0,0)=\{z(\cdot)\in Z:\  {\xi}''+\xi=z(t),\; \xi(0)=\xi(\pi)=0, \ \xi\in X\}.
$$
The general solution of the equation $\xi''(t)+\xi(t)=z(t)$ has the
form
\begin{equation}
\label{rozw-og}
\begin{array} {rcl}
\xi(t) &=&C_{1}\cos t +C_{2}\sin t -\sin t \displaystyle\int_{0}^{t}\cos \tau
z(\tau) d\tau \\
&& + \cos t \displaystyle\int_{0}^{t}\sin \tau z(\tau) d\tau,\quad
 C_{1},C_{2}\in \mathbb{R}.
\end{array}
\end{equation}
By using the boundary conditions, we obtain $C_{1}=0$ and
$\displaystyle\int_{0}^{\pi}\sin \tau z(\tau)d\tau =0.$ Finally, we get
$$Z_{1}=\hbox{Im} F_y'(0,0)=\{z(\cdot)\in Z:\int_{0}^{\pi}\sin \tau
z(\tau)d\tau =0\}\neq Z.$$ 

The kernel of the operator $F_y'(0,0)$ is
defined by solving the following boundary-value
 problem
$$\xi''+\xi=0,\; \xi(0)=\xi(\pi)=0 .$$ By \eqref{rozw-og}, $\xi(t)=C\sin t$ is a solution, so that
$W_{2}=\Ker   F'_{y}(0,0)=\hbox{span}\{\sin t\}$ and
$$
P_{W_{2}}z=\dfrac{2}{\pi}\sin t \int_{0}^{\pi}\sin \tau
z(\tau)d\tau, \;\;z\in Z.
$$
By an assumption \eqref{gg0}, we have $g'(0)=0$ and $$F''(0, 0)=g''(0),
\ldots, F^{(p)}(0, 0)=g^{(p)}(0).$$ Finally we get
\begin{equation}
\begin{split}
&Y_{2}=\hbox{span}(\hbox{Im} P_{Z_{2}}F''(0, 0)[\cdot])\\
&=\hbox{span}\left\{y(t): \exists h\in X:y(t)=\dfrac{2}{\pi}\sin t
\int_{0}^{\pi}\sin \tau g''(0)[h(\tau)]^{2}d\tau\right\}.
\end{split}
\end{equation}
With additional assumptions, 
$$F_{yy}''(0,0)=g''(0),\ldots, F^{(p)}_{y\ldots y} (0)
\ldots, F^{(p)}(0,0)=g^{(p)}(0)=0,$$ we get that other spaces $Y_{3},\ldots,Y_{p}$ can be defined in a similar way
and depend only on the mapping $g(x).$ Next we define mappings:
$$f_{1}(x)=F(x), \;f_{i}(x)=P_{Y_{i}}F(x), \quad i=2,\ldots , p.$$
Moreover, note that
$f_{i}(0)[\cdot]^{i-1}=P_{Y_{i}}g^{(i)}(0)[\cdot]^{i-1}.$ Then,
the $p$-factor operator has the following form
$$
\Psi_{p}(h)(\cdot)=(\cdot)''+(\cdot)+\dfrac{1}{2}
P_{Y_{2}}g''(0)[h]+\ldots +
\dfrac{1}{p!}P_{Y_{p}}g^{(p)}(0)[h]^{p-1}.
$$
\begin{example}
Consider the second-order nonlinear ordinary differential equation
\begin{equation}
\label{5eq23}
    x''(t) + x(t) + x^{2}(t)=a \sin t, \;\; x(0)=x(\pi)=0,
\end{equation}
 where $a\in \mathbb{R}.$ We have $g(x)=x^{2},$ $v(t)=a \sin t.$ Let
us verify that for $x^{0}=\mathbf{0}$ all assumptions of
\ref{5th18} are satisfied for the
mapping $F(x)=x''+x+x^{2}-a \sin t$ with a sufficiently small $a\geq 0$ and $p=2.$\\
It is easy to see that $F(x^{0})=-a \sin t$ and hence from
\ref{5eq16} $\delta=a.$ Moreover, $F$ is singular at $x^{0}.$ In
this case, $2$-factor operator has the following form
$$\Psi_{2}(h(t))(\cdot)=(\cdot)''+ (\cdot)+\dfrac{2}{\pi} \sin t
\int_{0}^{\pi}h(\tau)(\cdot)\sin \tau d\tau$$ and the
corresponding nonlinear operator is
$$\Psi_{2}[h(t)]^{2}=h''(t)+h(t)+\dfrac{2}{\pi}\sin t \int_{0}^{\pi}[h]^{2}(\tau)\sin
\tau d\tau = \dfrac{2}{\pi}\sin t \int_{0}^{\pi}[h]^{2}(\tau)\sin
\tau d\tau,$$ since $h(t)\in \Ker   F'(0)=\hbox{span}\{\sin t\}.$ From
the Banach condition in \ref{GIFT}, we obtain
$$
h\in \left\{x\in X: \Psi_{2}[x]^{2}=-F(x^{0})\right\}
$$
and then
$$h(t)=\pm\dfrac{\sqrt{3\pi a}}{2\sqrt{2}} \sin t, \hbox{ for } a>0, \;\; \hat{h}(t)=\pm\sin t.
$$
The $p$-factor operator for $\hat{h}(t)=\sin t$ has the form\\
$$\Psi_{2}(\hat{h}(t))(\cdot)=(\cdot)''+(\cdot)+\dfrac{2}{\pi}\sin
t \int_{0}^{\pi}(\cdot)\sin^{2} \tau d\tau .$$

Then  mapping  $F$ is $2$-regular at the point $x^{0}$ along $\hat{h}$
because $\;\Psi_{2}(\hat{h}(t))$ is a surjection.

Now, taking into account \ref{5eq16}--\ref{5eq21} for the considered
boundary-value problem, we get $c_{1}=\dfrac{1}{\pi},$ $c_{2}=2$
and $c_{3}=\sqrt{\dfrac{\pi}{96}}.$ Finally, applying 
\ref{5th18}, we deduce that for $a\leq \dfrac{\pi}{9600}$ there is an 
$\varepsilon$ (for instance, $\varepsilon=\dfrac{\pi}{80}$), such
that there exists a solution $\bar{x}(t)\in
B_{\varepsilon}(x^{0}),$ i.e.
$\bar{x}''(t)+\bar{x}(t)+\bar{x}^{2}(t)=a \sin t.$
\end{example}

\subsection{Interpolation by polynomials}
\label{sub4.11.1}

In this section, we consider one of the newest applications of the $p$-regularity theory. There are many books on numerical analysis and numerical methods where the topic of interpolation and polynomial approximation is described in detail (see, for example,  \cite{BB} and \cite{BF}).

\subsubsection{Newton interpolation polynomial}

Let $f$ be $\mathcal{C}^{p+1}([a,b])$ and consider the equation $$f(x)=0,$$ where $x\in [a,b]$. For some $\Delta x>0$, define points $x_i$, $i=0, \ldots, n$, as follows:
$$
x_0=a, \; x_1=x_0 + \Delta x, \; x_2 = x_1 + \Delta x, \ldots,  x_n=b.$$ Let $$y_i=f(x_i), \quad i=0,1,\ldots, n.$$

The problem of interpolation is to find a polynomial $W_n(x)$ of degree at most $n$, such that $W_n(x_i)=y_i$, $i=0,\ldots,n$, and that gives a good approximation of the function $f(x)$.

Let $\varepsilon = \Delta x$ be sufficiently small and assume  that $|f(x)-W_n(x)|\leq C_1 \varepsilon$, where $C_1 \geq 0$ is a constant.
Assume that the equation $f(x)=0$ has a solution $\bar{x} \in (a, b)$ and the equation $W_n(x)=0$ has a solution  $\tilde{x} \in (a, b)$.
Our goal is to use the interpolation polynomial $W_n(x)$ and its solution $\tilde{x}$ to obtain the $\varepsilon^2$-accuracy of the solution $\bar{x}$  in the sense that
\begin{equation} \label{aq}
| \bar{x} - \tilde{x}| \leq C \varepsilon^2,
\end{equation}
where $C \geq 0$ is a constant. In the regular case, this can be obtained by using, for example, the Newton interpolation polynomial $W_n(x)$ with $ \Delta x= \varepsilon.$

Recall  that the Newton interpolation polynomial of degree $n$, related to the data points 
$$(x_0,y_0),(x_1,y_1),\ldots,(x_n,y_n),$$ is defined by
$$
  W_n(x) 
   = \begin{array}[t]{l}\alpha_0+\alpha_1(x-x_0)+\alpha_2(x-x_0)(x-x_1)+\\
   \ldots +\alpha_n(x-x_0)(x-x_1)\ldots (x-x_{n-1})\\
   = \mathop{\sum}\limits_{k=0}^n \alpha_k \omega_k(x),
   \end{array}
$$
where
\begin{equation}  \label{om}
\begin{array} {l}
\omega_0(x)=1,\\
\omega_i(x)=(x-x_0)(x-x_1)\ldots (x-x_{i-1})=\prod_{j=0}^{i-1}(x-x_j), \quad i=1,\ldots,n.
\end{array}
\end{equation}

The coefficients $\alpha_k$ are called {\em divided differences} and defined using the following relations:
\begin{equation} \label{alp}
\alpha_k=  [y_0, \ldots, y_k], \quad  k=0,1,\ldots,n,
\end{equation}
 where
\begin{eqnarray*}
[y_k] &&= y_k,\; k=0, \ldots, n,  \\[2mm]
 [y_k, \ldots, y_{k+j}]
& &= \frac{[y_{k+1},\ldots, y_{k+j}]- [y_k, \ldots , y_{k+j-1}]}{x_{k+j}-x_k}],\\
&& \qquad k=0, \ldots,n-j, \quad j = 1, \ldots, k. 
\end{eqnarray*}

In the following example, we consider a nonlinear function $f(x)$, which is not regular  at a solution of the equation $f(x) = 0$, and investigate
if a solution of the equation $W_n(x) = 0$ provides desired accuracy \eqref{aq} for the solution $\bar{x}$ of $f(x) =0$, assuming that $|f(x)-W_n(x)|\leq C_1 \varepsilon$ holds for a sufficiently small $\varepsilon$.

\begin{example}
 \label{ex:interp}
Consider function $f(x)=x^3$. The solution of the equation $f(x)=0$ is $\bar{x}=0$. Function $f(x)$ is  singular at $\bar{x}=0$ up to the second order because $f^{(i)}(0)=0$, $i=1,2.$ The goal in this example is to investigate whether estimate \eqref{aq} is satisfied when we are using the interpolation polynomial $W_1(x)$ and a solution of $W_1(x)=0$  to approximate the solution of $f(x)=0$. Using the equations given above with $n=1$, we get
$$W_1(x)=\alpha_0+\alpha_1(x-x_0),$$
where the coefficients $\alpha_0$ and $\alpha_1$ are determined by using Equation \eqref{alp}.

 Let $\varepsilon = \Delta x$ be sufficiently small and  consider the segment $[a, b] = [-\dfrac{1}{3} \varepsilon, \dfrac{2}{3} \varepsilon]$.
The interpolation points are $x_0=a=-\dfrac{1}{3} \varepsilon$ and $x_1=b=\dfrac{2}{3} \varepsilon$. Calculating the coefficients, we get $$\alpha_0=f(x_0)=-\dfrac{\varepsilon^3}{27} \quad \mbox{and} \quad \alpha_1=\frac{f(x_1)-f(x_0)}{x_1-x_0}=\dfrac{\varepsilon^2}{3}.$$ 
 Hence, the interpolation polynomial has the form
$$W_1(x)=-\dfrac{\varepsilon^3}{27}+\dfrac{\varepsilon^2}{3} (x + \dfrac13 \varepsilon)
= -\dfrac{\varepsilon^3}{27} + \dfrac{\epsilon^2}{3}x + \dfrac{\varepsilon^3}{9}=\dfrac{2 \varepsilon^3}{27}+ \dfrac{\epsilon^2}{3}x ,$$
and, moreover, $$|W_1(x) - f(x)| \approx \varepsilon^3\leq C_2\varepsilon, \quad C_2 \geq 0,$$
for a sufficiently small $\varepsilon$.

The solution of the equation $W_1(x)=0$ is $$\tilde{x}=-\frac{2}{9}\varepsilon,$$
which is not satisfactory from the approximation accuracy point of view since $$|\tilde{x} - \bar{x}| = \left|-\frac{2}{9} \varepsilon - 0 \right|\approx \varepsilon>\varepsilon^2$$
and the desired accuracy \eqref{aq} is not obtained.

Thus, in the degenerate case, contrary to the regular case, while we have the required accuracy of the approximation for the function $f$, the accuracy of the solution is only of the order $\varepsilon$, and not  $\varepsilon^2$.
\end{example}

\subsubsection{The $p$-factor interpolation method}
\label{sub4.11.2}

In this section, we will demonstrate that the desired accuracy  \eqref{aq} of the solution of the equation $f(x)=0$ in the degenerate case can be obtained by using the $p$-factor interpolation polynomial instead of the classical Newton interpolation polynomial for the purpose of finding an approximate solution of $f(x) =0$.

Let $f:\mathbb{R}\rightarrow \mathbb{R}$ be a $\mathcal{C}^{p+1}$ function that is singular at a point $\bar{x}$.

For some $p>1$, we associate $f$ with the $p$-factor function $\bar{f}$ of the form
$$\bar{f}(x)=f(x)+f'(x)h+\cdots+f^{(p-1)}(x)[h]^{p-1},$$
where $h\in \mathbb{R}$, $h\neq 0$.
Similarly to the Newton interpolating method, we construct the $p$-factor interpolation polynomial $\bar{W}_n(x)$ using function $\bar{f}$ as follows:
$$
\bar{W}_n(x) = \mathop{\sum}\limits_{k=0}^n \bar{\alpha}_k \omega_k(x),
$$
where functions $\omega_k(x)$ are defined in the same way as in \eqref{om} and the coefficients $\bar{\alpha}_k$, $k=0,1,\ldots,n$, are determined by
$$
\begin{array} {l}
\bar{\alpha}_0=[\bar{y}_0]=\bar{f}(x_0), \quad 
[ {\bar{y}_1}] = \bar{f}(x_1), 
\\[2mm]
\bar{\alpha}_1=[\bar{y}_0, \bar{y}_1]=\dfrac{[\bar{y}_1]-[\bar{y}_0]}{x_1-x_0},\\
\vdots\\
\bar{\alpha}_n=[\bar{y}_0, \ldots, \bar{y}_n] =\dfrac{[\bar{y}_1, \ldots,\bar{y}_n] - [\bar{y}_0, \ldots,\bar{y}_{n-1}]}{x_n-x_0}.
\end{array}
$$

\begin{theorem}
  Let the equation $f(x)=0$ has a solution $\bar{x} \in (a, b)$. Assume that $f\in \mathcal{C}^p([a,b])$ be $p$-regular along $h\neq 0$ at the point $\bar{x}$. Assume that 
 $\bar{W}_n(x)$ is the Newton interpolation polynomial for the function  $\bar{f}$,
  such that  $\varepsilon= \Delta>0$ is
   a  sufficiently small interpolation step. 
 
 Then the equation
   $\bar{W}_n(x)=0$ has a solution $\hat{x}\in (a,b)$
  such that 
  $$|\hat{x}-\bar{x}|\leq c \varepsilon^2,$$
 where $c>0$ is independent constant.
\end{theorem}

We omit the proof, as it is similar to the proof of convergence of the classical iterative Newton method.

Similarly to the definitions in the previous sections, $f\in \mathcal{C}^p([a,b])$ is $p$-regular along $h\neq 0$ at the point $\bar{x}\in (a,b)$ if there is a natural number $p\geq 2$ such that
$$
f^{(i)} (\bar{x}) =0, \quad i=1, \ldots, p-1, \quad \mbox{and} \quad f^{(p)} (\bar{x}) \neq0.
$$
Note that if $p=1$, the definition of $p$-regular function $f$ reduces to the definitioin of the regular function and the  $p$-factor interpolation polynomial $\bar{W}_n(x)$ becomes the Newton interpolation polynomial   ${W}_n(x)$.

\vspace{2mm}

{\it Example \ref{ex:interp} (continued.)} To use the $p$-factor interpolation method, define function $\bar{f}$ with $p=2$ and $h=1$ as
$$\bar{f}(x)=f(x)+f'(x)h+f^{\prime \prime}(x)[h]^{2}= x^3 + 3x^2 + 6x.$$
Now we can define  the $p$-factor interpolation polynomial $\bar{W}_1 (x)$. We use the same segment as above so that  the interpolation points are $x_0=a=-\dfrac{1}{3} \varepsilon$ and $x_1=b=\dfrac{2}{3} \varepsilon$. 
Define the coefficients: 
$$\bar{\alpha}_0= \bar{f}(x_0) = -\dfrac{1}{27} \varepsilon^3 + \dfrac{1}{3} \varepsilon^2 - 2 \varepsilon$$
and 
$$ 
\bar{\alpha}_1=\dfrac{ \bar{f}(x_1) - \bar{f}(x_0)}{x_1-x_0}= \dfrac{1}{3} \varepsilon^2 +  \varepsilon + 6
$$
The  $p$-factor interpolation polynomial is
\begin{eqnarray*}
\bar{W}_1 (x) &= & \bar{\alpha}_0  + \bar{\alpha}_1 (x - x_0) \\
&=& -\dfrac{1}{27} \varepsilon^3 + \dfrac{1}{3} \varepsilon^2 - 2 \varepsilon +  \left( \dfrac{1}{3} \varepsilon^2 +  \varepsilon + 6\right) \left(x+ \dfrac{1}{3} \varepsilon\right) \\
&=& \left(  \dfrac{1}{3} \varepsilon^2 +  \varepsilon + 6\right)x 
+\dfrac{2}{27} \varepsilon^3+ \dfrac{2}{3} \varepsilon^2 .
\end{eqnarray*}
Hence,
$$
\left| \bar{W}_1 (x) - f(x) \right| \leq C_3 \varepsilon, \quad C_3 \geq 0,
$$
for a sufficiently small value of $\varepsilon$.


Solving the equation $\bar{W}_1 (x) =0$,  we get $$\hat{x} = - \dfrac{\dfrac{2}{27} \varepsilon^3+ \dfrac{2}{3} \varepsilon^2}{\dfrac{1}{3} \varepsilon^2 +  \varepsilon + 6}
= - \dfrac{\dfrac{2}{9} \varepsilon^3+ 2 \varepsilon^2}{ \varepsilon^2 + 3 \varepsilon + 18} .
$$
Therefore,
 $$|\hat{x} - \bar{x}| = \left|-\dfrac{\dfrac{2}{9} \varepsilon^3+ 2 \varepsilon^2}{ \varepsilon^2 + 3 \varepsilon + 18} -0 \right|  < \dfrac{3\varepsilon^2}{18}=\dfrac{\varepsilon^2}{6},
$$
so we obtained the desired  $\varepsilon^2$-accuracy  \eqref{aq} of the solution of the equation $f(x)=0$. 

\vspace{2mm}

Now, let us take a look back and compare using the polynomials $W_1(x)$ and $ \bar{W}_1 (x)$ for obtaining the desired accuracy  \eqref{aq} of the solution $\bar{x}$ of $f(x)=0$ with the function $f$ from Example \ref{ex:interp}. As we mentioned above, the polynomial $W_1(x)$ is a good approximation for the function $f(x)$ because $$|W_1(x) - f(x)| \approx \varepsilon^3 \leq C_2\varepsilon, \quad C_2 \geq 0.$$ At the same time, the solution $\tilde{x}$ of  $W_1(x)=0$ does not give the desired accuracy of the solution because
$$
   \left|\tilde{x} -\bar{x} \right| \approx \varepsilon \geq C_4 \varepsilon, \quad C_4 \geq 0,
$$
and the desired accuracy of $\varepsilon^2$ is not reached.

Now, let us take a look at the $p$-factor interpolation polynomial $ \bar{W}_1 (x)$.  Using  the information from above, we get that $\bar{W}_1 $ approximates the function $f(x)$ with the order $\epsilon$:
$$
\left| \bar{W}_1 (x) - f(x) \right| \approx \varepsilon .
$$
Using the $p$-factor interpolation polynomial $ \bar{W}_1 (x)$, we get the desired accuracy for the solution $\bar{x}$ of the solutions of $f(x)=0$. As shown above, the solution $\hat{x}$ of $ \bar{W}_1 (x)=0$ satisfies 
estimate~\eqref{aq}:
 $$|\hat{x} - \bar{x} | \leq  \dfrac16 \varepsilon^2.
$$
Note this accuracy was not obtained using the classical interpolation polynomial $W_1(x)$.

\section{Conclusion}
\label{sec:C}

In this paper, we described various applications of the theory  of $p$-regularity. We should note that  we did not cover all areas where the results of the theory can be applied. In addition, there are other areas of mathematics, where the theory of $p$-regularity (or $p$-factor-analysis) has not been used yet. For example, we did not give examples of applying the theory to analysis of existence of solutions of singular nonlinear partial differential equations, such as Burger's nonlinear equation, Laplace nonlinear differential equation and others (see, for example, \cite{MedTr1}). We also have not covered results related to existence of solutions depending on a parameter for Van der Paul differential equation, Duffing equation and others (see, for example, \cite{MedTr2}). Other results not covered in this paper include examples of applying the theory of $p$-regularity for analysis of nonlinear dynamical systems (\cite{MedTr3}), optimality  conditions for optimal control problems in nonregular (degenerate) case (\cite{PrTr17}), and for inequality-constrained optimization problems (\cite{BrTr17}). Based on the theory of $p$-regularity, we can create the theory of so-called $p$-convexity, which can be effective for analysis of nonlinear problems.

The construction of $p$-factor-operator described in the paper presents a developed approach for analysis of degenerate situations of  various origin. The construction can be generalized using various vectors $h$.  For example,  in (\cite{BrTr03}) we gave the following generalization of the construction of $p$-regularity and of the $p$-factor-operator. It allowed us to derive the main result for description of nonlinear mappings, a generalization of Lyusternik Theorem.

Let for an element $h \in X$, $\|h\|=1$, the space $Y$  can be presented as
\begin{equation} \label{Y=sumY}
   Y = Y_1 (h) \oplus Y_2  (h),
\end{equation}
where $Y_1 (h) = {\rm Im} F^{(p)} (\bar{x}) [h]^{p-1}$ and $Y_2(h)$ is a closed complementary subspace of  $Y_1 (h) $ in $Y$. Let
$P_1$ and $P_2$ be operators of projection onto $Y_1 (h) $ and $Y_2 (h),$ respectively.

\begin{theorem}
\label{thL}
Let $X$ and $Y$ be Banach spaces,
$F\in \mathcal{C}^{p+1}(X),$ and
$$
   F^{(r)}(\bar{x}) = 0, \quad r=0,1,\dots,p-1 , \quad p \geq 2.
$$
Assume that condition \eqref{Y=sumY} is satisfied for $h \in   \Ker^{p} F\, ^{(p)} (\bar{x})$ and there exists an element $\tilde{h} \in $,
$\| \tilde{h} \| =1$, such that
$$
 F^{(p)}(\bar{x}) [h]^{p-1} [\tilde{h} ] =0,
$$
$$
  \| P_2 F (\bar{x} + th + t^*  \tilde{h} )\| \leq t^{(p-2) +2w+\epsilon}, \quad w \in (1, 3/2), \quad \epsilon \in (0, 1).
$$
\begin{equation} \label{Norm-Inv}
\left\| \left\{F^{(p)}(\bar{x}) [h]^{p-1} + P_2 F^{(p)}  [h]^{p-2}  [\tilde{h}] \right\}^{-1} \right\|  \leq \infty ,
\end{equation}
where $t \in (0, \delta)$ and $ \delta > 0$ is sufficiently small. Then $ h \in T_1 M(\bar{x})$.
\end{theorem}
In Theorem \ref{thL}, we have one of possible modifications of the $p$-factor-operator in the form \eqref{Norm-Inv}.

Another possible construction of $p$-factor-operator can be found in \cite{IzTr94}.

As we pointed out at the beginning of the manuscript, degenerate problems  and ill-posed problems are equivalent. Therefore, 
ill-posed problems need to be solved by  methods that are adapted for finding degenerate solutions. Those methods include $p$-factor-methods and approaches based on the theory of $p$-regularity. We plan to describe methods for ill-posed problems in our future publication.

\bibliographystyle{plain}
\bibliography{references}

\end{document}